\documentclass{amsart}

\usepackage{amsmath, amssymb, amsthm, amsrefs, pst-node, tikz-cd, mathtools}
\usepackage[left=3cm, right=3cm, top=3cm]{geometry}

\newtheorem{theorem}{Theorem}[section]
\newtheorem{lemma}[theorem]{Lemma}
\newtheorem{proposition}[theorem]{Proposition}
\theoremstyle{definition}

\theoremstyle{remark}
\newtheorem{remark}[theorem]{Remark}
\numberwithin{equation}{section}

\newcommand{\K}{\mathcal{K}}
\newcommand{\F}{\mathcal{F}}
\newcommand{\p}{\mathfrak{p}}
\newcommand{\fp}{\mathbb{F}_p}
\newcommand{\fpbar}{\overline{\mathbb{F}}_p}
\newcommand{\W}{\mathbb{W}}
\newcommand{\qp}{\mathbb{Q}_p}
\newcommand{\cp}{\mathbb{C}_p}

\newcommand{\art}{\mathbf{ART}}
\newcommand{\wphi}{W_{\Phi}}

\newcommand{\ord}{\text{ord}}
\newcommand{\cm}{\mathcal{CM}}
\newcommand{\kphip}{k_{\Phi,\p}}
\newcommand{\kphipbar}{\overline{{k}_{\Phi,\p}}}

\newcommand{\epf}{\epsilon_{\mathfrak{p}_F}}
\newcommand{\phispbar}{\overline{\phi^{sp}}}
\newcommand{\phispec}{\phi^{sp}}
\newcommand{\jphi}[1]{J_{\Phi_1}\mathcal{D}_1^{(k)}(R^{(#1)}) \rightarrow \mathcal{D}_2^{(k)}(R^{(#1)})/J_{\Phi_2}\mathcal{D}_2^{(k)}(R^{(#1)})}
\newcommand{\Z}{\mathcal{Z}}

\newcommand{\wmatrix}{\begin{bmatrix}
0 & -1 \\
1 & 0
\end{bmatrix}}
\newcommand{\nmatrix}{\begin{bmatrix}
1 & x \\
0 & 1
\end{bmatrix}}
\newcommand{\Hom}{\text{Hom}}

\newcommand{\D}{\mathbb{D}}
\newcommand{\m}{\mathfrak{m}}
\newcommand{\ktildebarp}{\overline{\tilde{k}}_\p}
\newcommand{\wtildep}{\tilde{\mathbb{W}}_{\p}}

\begin{document}

\title{Special Correspondences of CM abelian Varieties and Eisenstein Series}

\author{Ali Cheraghi}
\address{The Fields Institute, 222 College Street, Toronto, ON M5T 3J1}
\email{acheragh@fields.utoronto.ca}




\setlength{\parindent}{0cm}

\begin{abstract}

Let $\cm_{\Phi}$ be the (integral model of the) stack of principally polarized CM abelian varieties with a CM-type $\Phi$. Considering a pair of nearby CM-types (i.e. different in exactly one embedding) $\Phi_1, \Phi_2$, we let $X = \cm_{\Phi_1} \times \cm_{\Phi_2}$ and define arithmetic divisors $\Z(\alpha)$ on $X$ such that the Arakelov degree of $\Z(\alpha)$ is (up to multiplication by an explicit constant) equal to the central value of the $\alpha^{\text{th}}$ coefficient of the Fourier expansion of the derivative of a Hilbert Eisenstein series.

\end{abstract}

\maketitle
\tableofcontents

\section{Introduction}

A program envisioned in \cite{Kudla1997}, \cite{Kudla2003} was to relate algebro-geometrical data, namely the degrees of some special divisors (correspondences between moduli spaces) to the coefficients of Fourier expansion of the central value of the derivative of an Eisenstein series. This has been successful in some special cases including \cite{Howard2009} 
for the case of CM elliptic curves, \cite{Howard2012} for the case of CM abelian varieties with CM by a field containing an imaginary quadratic field. In this paper, we consider general CM-fields (i.e. not necessarily containing an imaginary quadratic field) and CM abelian varieties having complex multiplication by this field and prove that degrees of some special divisors (a correspondence between the moduli space of CM abelian varieties with a fixed CM-type and the moduli space of CM abelian varieties with another fixed CM-type) are equal (up to a constant factor) to the Fourier coefficients of a Hilbert Eisenstein modular form.

Let $K$ be a CM-field with maximal totally real subfield $F$. Let $O_K$ be the ring of integers of $K$. We denote by $\tilde{K}$ the normal closure of the field $K$ and for a prime $\p$ of $\tilde{K}$, $\tilde{k}_{\p}$ be the residue field of $\tilde{K}$ at $\p$ and $\ktildebarp$ a choice of algebraic closure of it. Also we suppose that $K/\mathbb{Q}$ has a ``mild ramification condition'' for small primes. If $(A, \lambda_A)$ is an abelian variety with CM by $O_{K}$ with the $O_K$-principal polarization $\lambda_A$, $(B, \lambda_B)$ be an $O_K$-principally polarized abelian variety having action by $O_{K}$ and $O_{K}$-principal polarization $\lambda_B$, we make $\Hom_{O_{K}}(A,B)$ into a Hermitian space by letting $\langle f,g \rangle := \lambda_A^{-1} \circ g^{\vee} \circ \lambda_B \circ f$ where $g^{\vee}: B^{\vee} \rightarrow A^{\vee}$ is the dual of $g$. This Hermitian form is $O_{K}$-valued. We define a
Deligne-Mumford stack $\Z(\alpha)$ of triples $(A,B,f)$ where $A$ is an abelian scheme with CM by $O_{K}$ and 
$B$ is an $O_K$-principally polarized abelian scheme 
with CM by $O_K$ and $f \in \Hom_{O_{K}}(A,B)$ is 
an $O_{K}-$linear homomorphism such that $\langle f,f \rangle = \alpha$. We compute the Arakelov degree of this stack (as a stack over $\text{Spec} \tilde{O}$  (ring of integers of $\tilde{K}$)) defined as follows:
$$\widehat{\deg} \Z(\alpha) = \frac{1}{[\tilde{K}:\mathbb{Q}]} \sum_{\p \subset \tilde{O}} \log\; N(\p) \sum_{z \in \Z(\alpha)(\ktildebarp)} 
\frac{\text{length}(O^{\text{\'et}}_{\Z(\alpha),z})}{ \# \text{Aut}\; z} $$ 
where $O^{\text{\'et}}_{\Z(\alpha),z}$ is the strictly Henselian local ring at $z$. As CM points have a transitive action of a group, it follows that for a fixed prime $\p$ of $\tilde{O}$, above length is constant for all $z$ and the problem reduces to that of computing the length for just one point and $\# \Z(\alpha)(\ktildebarp)$. Let $\mathfrak{p}_F$ be the prime of $F$ below $\p$. For an fractional ideal $I$ of $O_F$ let $\rho(I)$ be the number of fractional ideals of $O_K$ with norm equal to $I$ in $F$, we will prove:
\begin{theorem}
\text{(Theorem 3.10)} Let $\alpha \in F^{\gg 0}$, assume that $K/F$ is not unramified at all finite primes, If $\mathfrak{p}$ is a prime such that $\mathfrak{p}_F$ is nonsplit, then
$$\# \Z(\alpha)(\ktildebarp) =\sum_{(A_1,A_2,f) \in Z(\alpha)(\ktildebarp)} \frac{1}{\# \text{Aut}(A_1,A_2,f)} = \frac{|C_K|}{w(K)} \rho(\alpha \mathfrak{p}_F^{-\epf} O_F)$$
\end{theorem}

After this theorem, to compute the lengths of the local rings $O^{\text{\'et}}_{\Z(\alpha),z}$, we relate it (using Serre-Tate theorem) to  ``CM $p$-divisible groups" defined over local rings and then use Grothendieck-Messing deformation theory.

\begin{theorem}
(Theorem 3.11) Let $\alpha \in F^{\gg 0}$, $\mathfrak{p}$ be a prime of $\tilde{K}$ such that $\mathfrak{p}_F$ is nonsplit in $K$, then at $z \in \mathcal{Z}(\alpha)(\ktildebarp)$:
$$length(O^{\text{\'{e}}t}_{\mathcal{Z}(\alpha),z}) = \frac{1}{2} e_{\mathfrak{p}_F}(\text{ord}_{\mathfrak{p}_F}(\alpha) + 1)$$
where $e_{\mathfrak{p}_F}$ is the ramification index of $\phispec(\mathfrak{p}_F)$ in $\tilde{K}/F$. 
\end{theorem}

The {\bf main result} is the relation of this degree to the derivative of the Fourier coefficients of an $\text{SL}_2(F)$-Eisenstein series of parallel weight $(1,1,\cdots,1)$.

\begin{theorem}
(Theorem 5.2) Suppose that the following ramification conditions are satisfied:\\
1) $K/F$ is not unramified at all finite primes.

2) For every rational prime $l \leq \frac{[\tilde{K}:\mathbb{Q}]}{[K:\mathbb{Q}]}+1$, ramification index is less than $l$,
then we have
$$\widehat{\deg}\widehat{\mathcal{Z}}(\alpha) = \frac{-|C_K|}{w(K)} \frac{\sqrt{N_{F/\mathbb{Q}}(d_{K/F})}}{2^{r-1} [K:\mathbb{Q}]} b_{\Phi}(\alpha,y)$$
where $|C(K)| = |\widehat{O}_F^{\times \gg 0}/N_{K/F}\widehat{O}_K^{\times}| h(K)$ where $h(K)$ is the class number of $K$, $w(K)$ is the number of roots of unity in $K$, $d_{K/F}$ is the relative discriminant of $K/F$, $r$ is the number of places (including archimedean) ramified in $K$, and $b_{\Phi}(\alpha,y)$ is the  $\alpha^{\text{th}}$ coefficient of the Fourier expansion of the derivative of a Hilbert Eisenstein series.
\end{theorem}

Before this result, a special case was considered in \cite{Howard2012} where $K$ contains a quadratic imaginary field $K_0$. Let $K$ be a CM-field containing a quadratic imaginary field $K_0$ with maximal totally real subfield $F$ and $\Phi$ be a CM-type of $K$. By the signature of $\Phi$, we mean a pair of numbers $(r,s)$ such that $r$ is the number of elements of $\Phi$ that restrict to identity on $K_0$ and $s$ is the rest. Let $O_K, O_{K_{0}}$ be the ring of integers of $K, K_0$, respectively. We denote by $K_{\Phi}$ the reflex field of $\Phi$ and for a prime $\p$ of $K_{\Phi}$, $\kphip$ be the residue field of $K_{\Phi}$ at $\p$ and $\kphipbar$ a choice of algebraic closure of it. Also suppose that $F, K_0$ have relatively prime odd discriminants so in particular $O_K = O_{K_0} \otimes O_F$. If $(A, \lambda_A)$ is an elliptic curve with CM by $O_{K_0}$ with the principal polarization $\lambda_A$, $(B, \lambda_B)$ be an $O_K$-principally polarized abelian variety having action by $O_{K_0}$ and $O_{K_0}$-principal polarization $\lambda_B$, we turn $\Hom_{O_{K_0}}(A,B)$ into a Hermitian space by letting $\langle f,g \rangle := \lambda_A^{-1} \circ g^{\vee} \circ \lambda_B \circ f$ where $g^{\vee}: B^{\vee} \rightarrow A^{\vee}$ is the dual of $g$. This Hermitian form is $O_{K_0}$-valued, we can also define $\langle\; ,\; \rangle_{CM}$ to be the unique $K$-valued Hermitian form with the property that $\text{tr} \langle\; ,\; \rangle_{CM} = \langle \;,\; \rangle$. In the aforementioned paper, Howard defines a 
Deligne-Mumford stack $\Z(\alpha)$ of triples $(E,A,f)$ where $E$ is an elliptic curve (over 
the assigned scheme) with CM by $O_{K_0}$ and 
$A$ is an $O_K$-principally polarized abelian variety 
with CM by $O_K$ and $f \in \Hom_{O_{K_0}}(E,A)$ is 
an $O_{K_0}-$linear homomorphism such that $\langle f,f \rangle_{CM} = \alpha$. He computes the Arakelov degree of this stack (as a stack over $\text{Spec} O_{\Phi}$  (ring of integers of $K_{\Phi}$)) defined as follows:
$$\widehat{\deg} \Z(\alpha) = \frac{1}{[K_{\Phi}:\mathbb{Q}]} \sum_{\p \subset O_{\Phi}} \log\; N(\p) \sum_{z \in \Z(\alpha)(\kphipbar)} 
\frac{\text{length}(O^{\text{\'et}}_{\Z(\alpha),z})}{ \# \text{Aut}\; z} $$ 
where $O^{\text{\'et}}_{\Z(\alpha),z}$ is the strictly Henselian local ring at $z$. As CM points have a trasitive action of a group, it follows that for a fixed prime $\p$ of $O_{\Phi}$, above length is constant for all $z$ and the problem reduces to that of computing the length for just one point and $\# \Z(\alpha)(\kphipbar)$.

The main result of \cite{Howard2012} is the relation of this degree to the derivative of the Fourier coefficients of an $\text{SL}_2(F)$-Eisenstein series of parallel weight $(1,1,\cdots,1)$. Suppose that we have the Fourier expansion
$$E(\tau,s) = \sum_{\alpha \in F} a_{\alpha}(s,y) e^{2\pi i \text{Tr}_{F/\mathbb{Q}}(\alpha \tau)}$$
for this Eisenstein series where $x+iy = \tau \in \mathfrak{h}^{[F:\mathbb{Q}]}$ where $\mathfrak{h}$ is the upper-half plane and $s \in \mathbb{C}$, then

\begin{theorem}
(\cite{Howard2012}, Theorem 4.2.3)
For $\alpha \in F$ totally positive, we have the following equality:
\begin{equation}
\label{1} 
\widehat{\deg}\Z(\alpha) = -\frac{h(K_0)}{w(K_0)} \frac{\sqrt{N_{F/\mathbb{Q}}(d_{K/F})}a^{\prime}_{\alpha}(0,y)}{2^{r-1}[K:\mathbb{Q}]}
\end{equation}
where $w(K_0)$, $N(d_{K/F})$, $r$ are defined as the above  theorem.
\end{theorem}

The way to prove our main result is to compute both the coefficients and degrees separately and see that they match. To compute the coefficients we use Yang's formulas \cite{Yang2013}, to compute the degrees we use Serre-Tate and Grothendieck-Messing theory. Indeed, for a pair of CM abelian varieties, we consider their $p$-divisible groups (so that these $p$-divisible groups have actions by the ring of integers of a local field) and then we will compute explicitly the condition on lifting of these $p$-divisible groups and maps between them.

In section 2, we consider CM p-divisible groups and determine the liftings of them that will be used to compute the lengths of local rings appearing in the Arakelov degree. In section 3, we define the moduli stack and special divisors and compute the Arakelov degrees of them. In section 4, we extend the special divisors to arithmetic Chow group to be able to define more general special divisors that will be related to negative part of Fourier series of the Hilbert Eisenstein series. In section 5, we define the Eisenstein series that will be related to these Arakelov degrees and finish the proof of the main theorem.

{\textit{Acknowledgements}.} This is part of author's thesis and I would like to thank my advisor, Stephen Kudla, for continuous helpful conversations and insights. I would also like to thank Ben Howard  for helpful correspondence.

\section{Local computation}
         
         In this section we will compute the length of the local ring appearing in the Arakelov degree by local considerations and use of Grothendieck-Messing theory. Here we consider the two cases of different CM-types and same CM-types, because when we change two abelian varieties by their $p$-divisible groups, the $p$-adic CM-types (to be defined below) might be equal or have exactly one difference.

\subsection{$p$-divisible groups with different CM-types}

Let $p$ be a prime, $\fp$ be the field with $p$ elements and $\fpbar$ the algebraic closure of it, $\W = \W(\fpbar)$ be the Witt vectors over $\fpbar$ (equivalently, the completion of the ring of integers of maximal unramified extension of $\qp$), $\cp$ the completion of the algebraic closure of $\qp$, also in this paper by a scheme we always mean a locally Noetherian scheme and an algebraic stack means a Deligne-Mumford stack, $\F$ be an extension of $\qp$ of degree $n$, $\K/\F$ a quadratic field extension of local fields, $(O_\K,\p_\K), (O_\F,\p_\F)$ are their ring of integers and the nonzero prime ideals of $O_{\K}, O_{\F}$, respectively. $x \mapsto \bar{x}$ be the nontrivial automorphism of $\K$ over $\F$. If $\phi: \K \hookrightarrow \cp$ is an embedding, then $\bar{\phi}(x) = \phi(\bar{x})$. 

Suppose that $\Phi_1, \Phi_2$ be a pair of nearby $p$-adic CM-types (A $p$-adic CM-type consists of $n$ embeddings $\K \hookrightarrow \cp$ such that for an embedding $\phi: \K \hookrightarrow \cp$ either $\phi$ or $\bar{\phi}$ is in $\Phi$), i.e. $\# (\Phi_1 \cap \Phi_2) = n-1$, so there is a unique $\phi^{sp} \in \Phi_1$ such that $\overline{\phi^{sp}} \in \Phi_2$, Now
fix two triples $(A_1, \kappa_1, \lambda_1), (A_2, \kappa_2, \lambda_2)$ such that for $i=1,2$:\\
$\bullet$ $A_i$ is $p$-divisible group over $\fpbar$.\\
$\bullet$ $\kappa_i: O_\K \rightarrow \text{End} A_i$ satisfies $\Phi_i$-determinant condition, i.e. for $x \in O_\K$, $t-x$ acts on $\text{Lie}(A_i)$ with determinant $\prod_{\phi \in \Phi_i} (t - \phi(x))$.\\
$\bullet$ $\lambda_i: A_i \rightarrow {A}^{\vee}_i$ be an $O_\K$-linear isomorphism, i.e. for all $x \in O_{\K}$, $\lambda_i \circ x = \bar{x} \circ \lambda_i$.

First of all, $A_i$'s are supersingular (i.e. all slopes of their Dieudonn\'e modules are $\frac{1}{2}$) by proposition 2.1.1 of \cite{Howard2012}. Now we consider the $O_\K$-module $L(A_1,A_2)= \Hom_{O_\K}(A_1,A_2)$ of all $O_\K$-linear homomorphisms from $A_1$ to $A_2$ and define a Hermitian form on it:
$$ \langle f,g \rangle = \lambda_1^{-1} \circ {g}^{\vee} \circ \lambda_2 \circ f \in \text{End}_{O_\K} A_1 = O_\K,\;\;\; f,g \in L(A_1,A_2)$$

Define $S = O_\K \otimes_{\mathbb{Z}_p} \W$, $\text{Fr} \in \text{Aut} \W$ be Frobenius, 
then $\text{Fr}$ acts on $S$ by $x \otimes w \mapsto x \otimes 
w^{\text{Fr}}$. Let $\K^{u}, \F^{u}$ be the maximal unramified extensions of $\qp$ inside $\K, \F$, respectively and $O^u_{\K}$, $O^u_{\F}$ their rings of integers. For each $\psi: O_\K^{u} \rightarrow \W$, there is an 
idempotent $e_{\psi} \in S$ such that $(x \otimes 1)e_{\psi} = (1 \otimes 
\psi(x))e_{\psi}$ for all $x \in O_\K^{u}$. They satisfy $
(e_{\psi})^{\text{Fr}} = e_{\text{Fr} \circ \psi}$. Also, that $S = 
\prod_{\psi} e_{\psi} S$ where $e_{\psi} S \cong \W_\K$ (ring of integers of the completion of maximal unramified extension of $\K$) is a DVR, so we 
also have the maps $\text{ord}_{\psi}: S \rightarrow \mathbb{Z}^{\geq 0} \cup \lbrace 0 \rbrace$. Put $m(\psi, \Phi_i) = \# \lbrace \phi \in \Phi_i : \phi|_{O_\K^{u}} = \psi \rbrace$. Let $\D(A_i)$ be the Dieudonn\'e module of $A_i$.  Here, we state lemma 2.3.1 of \cite{Howard2012}. 

\begin{lemma}
For $i= 1 ,2$, 
we have an isomorphism of $S$-modules $\D(A_i) \cong S$. $F,V \in \text{\normalfont End} \D(A_i)$ act on $S$ by $F = a_i \circ \text{\normalfont Fr}$, $V= b_i \circ \text{\normalfont Fr}^{-1}$,  for some $a_i, b_i \in S$ such that $a_ib_i^{\text{\normalfont Fr}} = p$ and $\text{ord}_{\psi} (b_i) = m(\psi, \Phi_i)$ for all $\psi: O_\K^{u} \rightarrow \W$.
\end{lemma} 

$L(A_1,A_2) \subseteq \text{Hom}_{O_\K \otimes_{\mathbb{Z}_p} \W}(\D(A_1),\D(A_2)) \cong \text{Hom}_S (S,S) \cong S$ as $S$-modules, so by the above $L(A_1,A_2) = \lbrace s \in S | (b_1s)^{\text{Fr}} = sb_2^{\text{Fr}} \rbrace$ (these are elements of $S$ that are compatible with $V$ (and so $F$)).

\begin{lemma}
\label{lemma2.2}
$\langle\; ,\; \rangle$ on $L(A_1,A_2)$ is identified with $\langle s_1,s_2\rangle = \xi s_1 \overline{s_2}$ (with identification as a subset of $S$) with $\xi \in S$ satisfying:\\
1) $\bar{\xi} = \xi$\\
2) $\xi S = S \;\;(\text{i.e. } \xi \in S^{\times})$\\
3) $(b_1\overline{b_1})^{\text{\normalfont{Fr}}} \xi = \xi^{\text{\normalfont{Fr}}} (b_2\overline{b_2})^{\text{\normalfont{Fr}}}$
\end{lemma}
\begin{proof}
By polarization we have $\W$-symplectic maps $\lambda_i: S \times S 
\rightarrow \W$ ($i = 1,2$) satisfying $\lambda_i(sx,y) = \lambda_i(x,\bar{s}y)$, $\lambda_i(F x,y) = \lambda_i(x,Vy)^{\text{Fr}}$
so we can find $\zeta_i \in S \otimes \mathbb{Q}$ such that $
\lambda_i(s_1,s_2) = \text{tr}_{\K/\qp} (\zeta_i s_1 \bar{s_2})$ such 
that $$\bar{\zeta} = -\zeta, \; p\zeta_i = (\zeta_ib_i
\overline{b_i})^{\text \normalfont{Fr}},\; \zeta_i S = \mathfrak{D}^{-1} S$$
where $\mathfrak{D}$ is the different of $\K$,
so that $\xi = \zeta_1^{-1} \zeta_2$ has the above properties.
\end{proof}
\begin{proposition}
For some $\beta \in \F^{\times}$ satisfying
\[\beta O_\K =  \begin{cases} 
      \mathfrak{p}_\F O_\K & \text{if}\; \K/\F\;  is \; \text{unramified} \\
      O_\K & \text{if}\; \K/\F\; is \; \text{ramified}\; \\ 
   \end{cases}
\]
we have $L(A_1,A_2) \cong O_\K$ with $\langle x,y \rangle = \beta x\bar{y}$ on $O_\K$.
\end{proposition}
\begin{proof}
First of all, as both $A_1,A_2$ are supersingular, we have an isogeny $A_2 \rightarrow A_1$, also that $\text{End} A_i = M_n(H)$, where $H$ is the quaternion division algebra over $\qp$. Now by Noether-Skolem, we can change this isogeny by some $h\circ f$ for $h \in M_n(H)$ so that it becomes $O_\K$-linear, so:
$$L(A_1,A_2) \otimes \qp \cong \text{Hom}_{O_\K}(A_1,A_1) \otimes \qp \cong \K$$
Consequently $L(A_1,A_2)$ is free of rank 1 over $O_\K$ (freeness is trivial from the definition).\\
Let $s$ be an $O_\K$-module generator of $L(A_1,A_2)$, such that $\beta = \xi s \bar{s}$ ($\xi$ is as in Lemma \ref{lemma2.2}), so that $\xi S = S$.\\
Now we have to determine $s\bar{s}S$.
Suppose that $d = [\K^{u}:\qp]$, and $$\lbrace \psi^0, \psi^1, \cdots, \psi^{d-1} \rbrace$$ be the set of embeddings $O_\K^u \rightarrow W$ such that $\text{Fr} \circ \psi^i = \psi^{i+1}$ (in a cyclic way). Now the relation $(b_1 s)^{\text{Fr}} = b_2^{\text{Fr}} s$ implies that
\begin{equation}
\ord_{\psi^{i+1}}(s) = \ord_{\psi^i}(s) - \ord_{\psi^i}(b_2) + \ord_{\psi^i}(b_1) =\\
\ord_{\psi^i}(s) - m(\psi^i, \Phi_2) + m(\psi^i, \Phi_1)
\end{equation}\\
Now we have two cases:\\
$\bullet$ If $\K/\F$ is ramified (same as $\K^u \subseteq \F$) then $x \mapsto \bar{x}$ acts trivially on $\K^u$, so both $\phi^{sp}$ and $\overline{\phi^{sp}}$ restrict to the same $\psi^i$ on $O_\K^u$ (call it $\psi^0$), so by the formula above, $\ord_{\psi^{i+1}}(s) = \ord_{\psi^i}(s)$, also $s$ is a generator, so $\ord_{\psi^i}(s) = 0$ for all $i$, so that $s \in S^{\times}$ and $\beta S = \xi s\bar{s} S = S$  \\
$\bullet$ If $\K/\F$ is unramified, then there are $i \neq  j$ not equal such that $\phi^{sp}|_{\K^u} = \psi^i,\; \overline{\phi^{sp}}|_{\K^u} = \psi^j$, now $\text{Fr}^{j-i}$ gives us the conjugation $x \mapsto \bar{x}$ on $\K^u$, as this is an involution, we have $j-i = \frac{d}{2} \mod d$, also as $s$ is a generator we have $\ord_\nu(s) = 0$ for some $\nu$ so that we can compute $\ord_{\psi^r}(s)$ for all $r$ to get the table above.

$$\ord_{\psi^r}(s\bar{s}) = \ord_{\psi^r}(s) + \ord_{\psi^{r+j-i}}(s) = 1$$
so $\beta S = s\bar{s} S =  \mathfrak{p}_\F S$
\end{proof}
Now suppose that $\tilde{\K}$ is a local field containing the normal closure of $\K/\qp$ and let $\tilde{\W}$ be the ring of integers of the completion of the maximal unramified extension of $\tilde{\K}$.
Let $\m \subset \tilde{\W}$ be the maximal, $R^{(k)} := \tilde{\W}/\m^k$, there's a unique deformation (by Grothendieck-Messing) $(A_i^{(k)}, \kappa_i^{(k)}, \lambda_i^{(k)})$ of $(A_i, \kappa_i, \lambda_i)$ to $R^{(k)}$. $$L^{(k)}(A_1,A_2) = \text{Im}(\text{Hom}_{O_\K}(A_1^{(k)}, A_2^{(k)}) \rightarrow \text{Hom}_{O_\K}(A_1,A_2)) \subseteq L(A_1,A_2)$$\\
From now on we assume the following ramification condition:\\
($\ast$) If $p \leq [\tilde{\K}:\phispec(\K)] + 1$ then the ramification index of $\tilde{\K}/\qp$ is less than $p$.
\begin{theorem}
\label{Thm2.4}
$f \in L(A_1,A_2)$ with $\langle f,f\rangle = \alpha$ can be lifted to $L^{(k)}$ but not to $L^{(k+1)}$, where 
$$k = \frac{1}{2}\ord_{O_{\tilde{\K}}}(\alpha \mathfrak{p}_\F)$$
\end{theorem}
We prove this by induction on the order of $f$ in $L(A_1,A_2) \cong O_\K$.

\begin{table}[]
\begin{tabular}{|l|l|l|l|l|l|l|l|l|l|l|l|l|}
\hline
$r$              & 0 & 1 & $\cdots$ & $i-1$ & $i$ & $i+1$ & $\cdots$ & $j-1$ & $j$ & $j+1$ & $\cdots$ & $d-1$ \\ \hline
$\ord_{\psi^r}s$ & 0 & 0 & 0        & 0     & 0   & 1     & 1        & 1     & 1   & 0     & 0        & 0     \\ \hline
\end{tabular}
\end{table}

First we prove a lemma about PD-thickening of $R^{(k)}$'s, let $\varepsilon$ be the ramification index of $\tilde{\K}/\phispec(\K)$.
\begin{lemma}
If $a \leq b \leq a+\varepsilon+1$, then $R^{(b)} \rightarrow R^{(a)}$ is a PD-thickening.
\end{lemma}
\begin{proof}
Its kernel is equal to $I = \m^a/\m^b$. To prove that this is a PD-thickening we have to (by definition) find some functions $\gamma_n: \m^a/\m^b \rightarrow \m^a/\m^b$ for $n > 0$ that behave like $\frac{x^n}{n!}$ such that they are the maps $\gamma_n : px \mapsto \frac{p^n}{n!}x^n$ when restricted to $\frac{(p)}{\m^b}$. \\
We know that $\m = (\pi)$ for some uniformizer $\pi$ of $\tilde{\K}$. Suppose that $(p) = (\pi)^{r\varepsilon}$ where $r$ is the ramification index of $\K/\qp$. Now we have two cases:\\
$\bullet$ $a \geq r\varepsilon$ then $I \subseteq (p)/\m^b$, so that we can put canonical divided powers $\gamma_n: px \rightarrow \frac{p^n}{n!}x^n$. Now to prove that $\gamma_n$  is valued in $I$, we compute the the order of its values (Suppose that $px \in I$):
$$\ord_{\pi}(\frac{p^n}{n!}x^n) = n\; \ord_{\pi}(px) - \ord_{\pi}(n!) \geq na - \ord_{\pi}n! = na - r\varepsilon(\left \lfloor{\frac{n}{p}}\right \rfloor + \cdots ) =$$ $$=a + (n-1)a - r\varepsilon(\left \lfloor{\frac{n}{p}}\right \rfloor + \cdots ) \geq a + r\varepsilon(n-1 - \left \lfloor{\frac{n	}{p}}\right \rfloor - \cdots) \geq a + r\varepsilon(n-1 - \frac{n-1}{p-1}) \geq a$$
$\bullet$ $a < r\varepsilon$, By the statement of the lemma and the case we are considering, we have
$$2r\varepsilon + 1 \geq r\varepsilon + \varepsilon + 1 > a + \varepsilon + 1 \geq b \implies 2r\varepsilon \geq b$$
so that
$$2r\varepsilon \geq b \Leftrightarrow \m^{2r\varepsilon}/(\m^b \cap \m^{2r\epsilon}) = 0 \Leftrightarrow ((p)/\m^b)^2 = 0 $$
so we can still use the divided power structure $\gamma_n: \pi x \mapsto \frac{\pi^n}{n!}x^n$, because on the ideal $(p)/\m^b$ we just need $\gamma_1$ which is the identity and $\gamma_i$ for $i \geq 2$ have to be zero maps on $(p)/\m^b$ (the fact that images of $\gamma_n$'s lie in $I$ uses ramification condition $(\ast)$).
\end{proof}
 First we prove the base case:
\begin{proposition}
If $f$ is an $O_\K$-module generator of $L(A_1,A_2)$ and $\mathfrak{D}_{\K/\F}$ be the relative different of $\K/\F$.\\
1) If $\K/\F$ is unramified, $f$ is in $L^{(\varepsilon)}(A_1,A_2)$ but not in $L^{(\varepsilon+1)}(A_1,A_2)$.\\
2) If $\K/\F$ is ramified, $f$ is in $L^{(k)}(A_1,A_2)$ but not in $L^{(k+1)}(A_1,A_2)$ where $k = \varepsilon\text{\normalfont ord}_{O_\K} \mathfrak{D}_{\K/\F}$. 
\end{proposition}
\begin{proof}
Let $J_{\Phi_i}$ be the kernel of the $\tilde{\W}$-algebra  map $O_\K \otimes_{\mathbb{Z}_p} \tilde{\W} \rightarrow \prod_{\phi \in \Phi_i} \cp(\phi)$ (sending $x \otimes 1$ to $(\phi(x))_{\phi \in \Phi_i}$) also let $\mathcal{D}_1,\mathcal{D}_2$ the Grothendieck-Messing crystals of $A_1,A_2$ respectively.\\
By the above lemma, the map $R^{(\varepsilon)} \rightarrow R^{(1)}$ is a PD-thickening so that:
$$\mathcal{D}_1(R^{(\varepsilon)}) = S \otimes_\W R^{(\varepsilon)}$$
$$\mathcal{D}_2(R^{(\varepsilon)}) = S \otimes_\W R^{(\varepsilon)}$$
by Theorem 2.1.3 of \cite{Howard2012}, Hodge filtrations of these are $J_{\Phi_1}\mathcal{D}_1(R^{(\varepsilon)})$ and $J_{\Phi_2}\mathcal{D}_2({R^{(\varepsilon)}})$, respectively, So $f$ lifts to $A_1^{(\varepsilon)} \rightarrow A_2^{(\varepsilon)}$ iff $$f: J_{\Phi_1}\mathcal{D}_1(R^{(\varepsilon)}) \rightarrow \mathcal{D}_2(R^{(\varepsilon)})/J_{\Phi_2}\mathcal{D}_2(R^{(\varepsilon)})$$ is trivial.

Now we prove parts 1 and 2 of the proposition:\\
1) Suppose that $f$ is $s \in S$ when writing $L(A_1,A_2)$ as a subset of $S$ as before, then we must prove that multiplication by $s$ is zero $\mod \m^{\varepsilon}$ in the first map below: 
\begin{equation}
 J_{\Phi_1}(S \otimes_\W \tilde{\W}) \xrightarrow{\text{s .}}
(S \otimes_\W \tilde{\W})/J_{\Phi_2}(S \otimes_\W 
\tilde{\W}) \xrightarrow{\overline{\phi^{sp}}} \tilde{\W}
 \end{equation}
 \label{eq2.2}
By lemma 2.1.2 of \cite{Howard2012}, we have that
$J_{\Phi_1}$ is generated by $$j(\psi,x) = e_{\psi}\prod_{\substack{\phi \in \Phi_1\\ \phi|{O_{\K}^u} = \psi} } (x \otimes 1 - 1 \otimes \phi(x))$$ for all $x \in O_\K$ and $\psi: O_\K^{u} \rightarrow \W$,
Now under multiplication by $s$, all $j(\psi,x)$ go to $J_{\Phi_2}(S \otimes_\W \tilde{\W})$ except for $j(\psi_0,x)$'s where $\psi_0 = \overline{\phi^{sp}}|_{O_\K^u}$ and $x \in O_\K$, these elements go to 
$$\phispbar(s) \prod_{\substack{\phi \in \Phi_1\\ \phi|{O_{\K}^u} = \psi_0} } (\phispec(\bar{x}) - \phi(x)) \in \tilde{\W}$$
after the last map. Now I claim that there's $a \in O_\K^{u}$ such that 
$$\prod_{\substack{\phi \in \Phi_1\\ \phi|{O_{\K}^u} = \psi_0} } (\phispec(\bar{a}) - \phi(a))$$
 is a unit in $\tilde{\W}$, indeed, by unramifiedness of $\K/\F$ (which is equivalent to $\K^u \nsubseteq \F$) we know that there exists $a \in O_\K^{u}$ such that $\bar{a} \neq a$, now 
$$\phi^{sp}(\bar{a})- \phi(a) = \psi_0(\bar{a}- a)$$
so that
$$j(\psi_0,a)= \psi_0(\bar{a}-a)^{m(\psi_0,\Phi_1)}$$
Thus $a-\bar{a}$ is a unit which implies that $j(\psi_0,a)$ is a unit.\\
Now we have, by the table in proposition 2.3, that $\ord_{\psi_0} s = 1$, and so that the map 2.2 is zero mod $\m^{\varepsilon}$ and nonzero mod $ \m^{\varepsilon+1}$, Now the map $$S \otimes_\W \tilde{\W} / J_{\Phi_2}(S \otimes_\W \tilde{\W}) \xrightarrow{(\phi(x))_{\phi \in \Phi_2}} \cp^n$$ is injective (because we killed the kernel), but the restriction of this map to image of multiplication by $s$ in equation \ref{eq2.2} sends $x$ to $(0,\cdots,0,\phispbar(x),0,\cdots,0)$, so that the image of $\phispbar$ determines everything and we can extend $f$ to $L^{(\varepsilon)}$.

Also, we have that $R^{(\varepsilon + 1)} \rightarrow R^{(1)}$ is a PD-thickening, so that the map $$f: J_{\Phi_1}\mathcal{D}_1(R^{(\varepsilon+1)}) \rightarrow \mathcal{D}_2(R^{(\varepsilon+1)})/J_{\Phi_2}\mathcal{D}_2(R^{(\varepsilon+1)})$$ is just equation \ref{eq2.2} $\mod \m^{\varepsilon + 1}$ which is nonzero so we are not able to extend $f$ to $L^{(\varepsilon+1)}$.\\
2) Suppose that we have extended $f$ to $L^{(k)}$,
in this case we will get that $f$ is equal to some $s \in S^{\times}$, by proposition 1, so $f: D(A_1) \rightarrow D(A_2)$ (these are Dieudonn\'e modules) is an isomorphism, If we check the $\Phi_i$-determinant condition on $D(A_1)/VD(A_1) \rightarrow D(A_2)/VD(A_2)$, we will get $\phi^{sp} = \overline{\phi^{sp}} \mod \m^k$, so $k \leq \varepsilon \text{\normalfont ord}_{O_\K} \mathfrak{D}_{\K/\F}$.\\
Now if $k \leq \varepsilon \text{\normalfont ord}_{O_\K} \mathfrak{D}_{\K/\F}$, then $f:A_1 \rightarrow A_2$ is an isomorphism of $p$-divisible groups with $O_\K$-action and respecting $\Phi_i$-determinant condition up to $\m^k$, now because deformations are unique and above map is an isomorphism we get an $O_\K$-linear map
$${f^{(k)}}: A_1^{(k)} \rightarrow A_2^{(k)}$$
that lifts $f$, so $f \in L^{(k)}(A_1,A_2)$ but not in $L^{(k+1)}(A_1,A_2)$. 
\end{proof}
Now we have the induction step in the next proposition:
\begin{proposition}
Let $\pi_{\K}$ be the uniformizer of $O_{\K}$. If $f \in L^{(k)}$, then $\pi_{\K} f \in L^{(k+\varepsilon)}$, also $\pi_{\K} .: L^{(k)}/L^{(k+1)} \rightarrow L^{(k+\varepsilon)}/L^{(k+\varepsilon+1)}$ is injective.
\end{proposition}
\begin{proof}
First of all, we have if $j \in J_{\Phi_1}$ then
$$(x \otimes 1 - 1 \otimes \overline{\phi^{sp}}(x))j \in J_{\Phi_2} = \ker(O_\K \otimes_{\mathbb{Z}_p} \tilde{\W} \rightarrow \prod_{\phi \in \Phi_2} \cp(\phi))$$
for all $x \in O_\K$, Henceforth, if we have an $O_\K \otimes_{\mathbb{Z}_p} \tilde{\W}$-linear map $f: J_{\Phi_1}M  \rightarrow N/J_{\Phi_2}N$ for $O_\K \otimes_{\mathbb{Z}_p} \wphi$-modules $M,N$, then
$$(x \otimes 1 - 1 \otimes \overline{\phi^{sp}}(x)) \circ f = 0$$ for all $x \in O_\K$.\\
Suppose that $f$ lifts to $f^{(k)}: A_1^{(k)} \rightarrow A_2^{(k)}$, then write $\mathcal{D}_1^{(k)}, \mathcal{D}_2^{(k)}$ for Grothendieck-Messing crystals of $A_1^{(k)}, A_2^{(k)}$, respectively. Now because we have the PD-thickening $R^{(k+\varepsilon)} \rightarrow R^{(k)}$,

\[ \begin{tikzcd}
J_{\Phi_1}\mathcal{D}_1^{(k)}(R^{(k+\varepsilon)}) \arrow{r}{f^{(k)}} \arrow[swap]{d}{} & \mathcal{D}_2^{(k)}(R^{(k+\varepsilon)})/J_{\Phi_2}(\mathcal{D}_2^{(k)}(R^{(k+\varepsilon)})) \arrow{d}{} \\%
J_{\Phi_1}\mathcal{D}_1^{(k)}(R^{(k)}) \arrow{r}{f^{(k)}}& \mathcal{D}_2^{(k)}(R^{(k)})/J_{\Phi_2}(\mathcal{D}_2^{(k)}(R^{(k)}))
\end{tikzcd}
\]
where the bottom row is 0 map, so that the top map becomes zero after $\otimes_{R^{(k+\varepsilon)}} R^{(k)}$, so its image is annihilated by $\m^\varepsilon$, and so $\pi_\K f^{(k)} = \phispbar(\pi_\K) f^{(k)}$ is zero on the top row and so it can be lifted to $L^{(k+\varepsilon)}$

Now suppose that $f \in L^{(k)}$ but not in $L^{(k+1)}$, we have the PD-thickening $R^{(k+\varepsilon + 1)} \rightarrow R^{(k)}$, so we have a map 
$$J_{\Phi_1}\mathcal{D}_{1}^{(k)}(R^{(k+\varepsilon+1)}) \rightarrow \mathcal{D}_2^{(k)}(R^{(k+\varepsilon+1)})/J_{\Phi_2}\mathcal{D}_2^{(k)}(R^{(k+\varepsilon+1)})$$
If $\pi_\K f^{(k)}$ lifts to $L^{(k+\varepsilon+1)}$, then we must have that $$\pi_\K f^{(k)}:J_{\Phi_1}\mathcal{D}_{1}^{(k)}(R^{(k+\varepsilon+1)}) \rightarrow \mathcal{D}_2^{(k)}(R^{(k+\varepsilon+1)})/J_{\Phi_2}\mathcal{D}_2^{(k)}(R^{(k+\varepsilon+1)})$$
is trivial, but by above $\pi_\K f^{(k)} = \overline{\phi^{sp}}(\pi_\K) f^{(k)}$, so 
$\phispbar(\pi_\K)f^{(k)}$ takes value in $$\m^{k+\varepsilon+1}\mathcal{D}_2^{(k)}(R^{(k+\varepsilon+1)})/J_{\Phi_2}\mathcal{D}_2^{(k)}(R^{(k+\varepsilon+1)})$$ so $f^{(k)}$ takes value in $\m^{k+1}\mathcal{D}_2^{(k)}(R^{(k+\varepsilon+1)})/J_{\Phi_2}\mathcal{D}_2^{(k)}(R^{(k+\varepsilon+1)})$ and so the map
$$f^{(k)}: \jphi{k+1} $$
is also trivial which means that $f$ can be lifted to $L^{(k+1)}$, which is a contradiction.
\end{proof}
Now we finish the proof of theorem 2:
\begin{proof}
$\beta$ be as before, $f = \pi_\K^{m}f_0$, $f_0$ an $O_\K$-module generator of $L(A_1,A_2)$. Suppose that $\K/\F$ is unramified, $f = \pi_\K^m f_0$ where $f_0$ is an $O_\K$-module generator of $\text{Hom}_{O_\K}(A_1,A_2)$, then \[\langle f_0,f_0 \rangle O_\K = \beta O_\K  =   
      \mathfrak{p}_\F O_\K
\]
Now by proposition 2, we have $f_0 \in L^{(\varepsilon)}$ but not in $L^{(\varepsilon+1)}$, so that $f \in L^{((m+1)\varepsilon)}$ but not in $L^{((m+1)\varepsilon + 1)}$.
Now $$\alpha = \langle f,f \rangle = \pi_\K^{2m} \langle f_0,f_0\rangle \implies \alpha O_\K = \pi_\K^{2m} \langle f_0,f_0\rangle O_\K = 
 \pi_\K^{2m}  
      \mathfrak{p}_\F O_\K
$$ $\implies k = (m+1)\varepsilon = \frac{1}{2}\text{ord}_{O_{\tilde{\K}}}(\alpha \mathfrak{p}_\F).$\\
Now If $\K/\F$ is ramified, then by proposition 2 we have that $f_0$ is in $L^{(\varepsilon\text{\normalfont ord}_{O_\K} \mathfrak{D}_{\K/\F})}$ but not in $L^{(\varepsilon\text{\normalfont ord}_{O_K} \mathfrak{D}_{\K/\F} + 1)}$, so that $f$ is in $L^{(\varepsilon(\text{\normalfont ord}_{O_\K} \mathfrak{D}_{\K/\F} + m))}$ and cannot be lifted more,
Here we have two cases:\\
$\bullet$ $p$ is odd: In this case, $\K/\F$ is tamely ramified and so $\text{\normalfont ord}_{O_\K} \mathfrak{D}_{\K/\F}$ is 1 and we get that the above implies  $$\alpha O_\K = \pi_\K^{2m} \langle f_0,f_0 \rangle O_\K = 
 \pi_\K^{2m}  
       O_\K
$$
 $\implies k = (m+1)\varepsilon = \frac{1}{2}\text{ord}_{O_{\tilde{\K}}}(\alpha \mathfrak{p}_\F).
$ \\
$\bullet$ $p = 2$, this case never happens by the ramification condition!
\end{proof}

\subsection{$p$-divisible groups with similar CM-types}
Here we have the same notation as before except that $\K$ is just quadratic \'etale extension over $\F$.

Now for another case of having two triples $(A_1,\kappa_1,\lambda_1),(A_2,\kappa_2,\lambda_2)$ $p$-divisible groups/$\fpbar$ with action of $O_\K$ and polarization, but now with the same CM-type $\Phi$. Let $\art$ be the category of local Artinian $\tilde{\W}$-algebras with residue field $\fpbar$.

\begin{proposition}
\label{prop2.8}
$R$ be an object of $\art$, $A_{i}^{\prime}$ be the unique deformation of $A_i$ to $R$, then
$$\text{Hom}_{O_\K}(A_1^{\prime}, A_2^{\prime}) \rightarrow \text{Hom}_{O_\K}(A_1,A_2)$$ is a bijection.
\end{proposition}
\begin{proof}
Let $f: S \rightarrow R$ be a surjection in $\art$ such that $(\text{Ker}f)^2 = 0$, so that we can define a trivial divided power structure on it, consequently if $f \in \text{Hom}_{O_\K}(A_1,A_2)$, let $\mathcal{D}_1, \mathcal{D}_2$ be Grothendieck-Messing crystals of $A_1,A_2$. As before let $J_{\Phi}$ be the kernel of the $\tilde{\W}$-algebra  map $O_\K \otimes_{\mathbb{Z}_p} \tilde{\W} \rightarrow \prod_{\phi \in \Phi} \cp(\phi)$ (sending $x \otimes 1$ to $(\phi(x))_{\phi \in \Phi}$), now we have 
$$J_{\Phi}(O_\K \otimes \tilde{\W}) \subseteq J_{\Phi}$$
(they are equal!), so that on
$$f: \mathcal{D}_1(S) \rightarrow \mathcal{D}_2(S)$$ 
Hodge filtrations map to each other ($J_{\Phi}$'s are Hodge filtrations by the Theorem 2.1.3 of \cite{Howard2012}), so by Grothendieck-Messing theory we can extend $f$ (uniquely) to a map $A_1^{S} \rightarrow A_2^{S}$ ($A_i^{S}$ means the deformation of $A_i$ to $S$), so we can do this inductively (induction on the length of $R$) to get the result.
\end{proof}

\section{Global computation}

In this section, we will define the special divisors on $\cm_{\Phi_1} \times \cm_{\Phi_2}$, compute the Arakelov degree of them with putting together a computation of the number of stacky points of the special divisors over residue fields and the lengths of local rings appearing in the Arakelov degree.

Let's fix some notation first. Let $F$ be a totally real extension of $\mathbb{Q}$ of degree $n$, $K$ be a CM-field having $F$ as its maximal totally real field, $O_K$, $O_F$ be the ring of integers of $K$, $F$ respectively, $\Phi_1, \Phi_2$ be a pair of nearby $p$-adic CM-types of $K$ (i.e. their intersection has $n-1$ elements and $\phi^{sp} \in \Phi_1$, $\overline{\phi^{sp}} \in \Phi_2$) and $\infty^{sp} := \phi^{sp}|_F = \overline{\phi^{sp}}|_F$. $\tilde{K}$ be the normal closure of $K/\mathbb{Q}$ (so it can be regarded as a subfield of $\mathbb{C}$), $\tilde{O}$ be the ring of integers of $\tilde{K}$. $\chi_{K/F}: \mathbb{A}_{F}^{\times}  \rightarrow \lbrace \pm 1 \rbrace$ be the quadratic character of the quadratic extension $K/F$. Let $\p$ be a prime of $\tilde{O}$ and $\p_K$ the prime in $O_K$ below $\p$ by the embedding $\phi^{sp}: K \hookrightarrow \tilde{K}$ (the use of $\phi^{sp}$ or $\overline{\phi^{sp}}$ does not matter in our case because we only care about nonsplit primes of $K/F$), $\tilde{k}_\p$
the residue field of $\tilde{K}$ at $\p$ and fix an algebraic closure $\overline{\tilde{k}}_\p$. $\tilde{\mathbb{W}}_\p$ be the ring of integers of the completion of the maximal unramified extension of $\tilde{K}_\p$.

Now we define the algebraic stack of CM abelian varieties and CM-type $\Phi$:
$\cm_{\Phi}$ is the algebraic stack over $\tilde{O}$ such that for an $\tilde{O}$-scheme $S$, $\cm_{\Phi}(S)$ is the groupoid of all triples $(A,\kappa,\lambda)$ with \\
$\bullet$ $A \rightarrow S$ is an abelian scheme of relative dimension $n$.\\
$\bullet$ $\kappa: O_K \rightarrow \text{End}A$ satisfies the $\Phi$-determinant condition.\\
$\bullet$ $\lambda: A \rightarrow A^{\vee}$ is an $O_K$-linear principal polarization.

$O_K$-linear polarization means that for all $x \in O_K$, we have $\lambda \circ x = \bar{x} \circ \lambda$.
Now for $(A_1,A_2) \in (\cm_{\Phi_1} \times_{\tilde{O}} \cm_{\Phi_2})(S)$, let $L(A_1,A_2) := \Hom_{O_K}(A_1,A_2)$ with the $O_K$-valued Hermitian form defined by 
$$f,g \in L(A_1,A_2); \;\;\; \langle f,g \rangle := \lambda_{A_1}^{-1} \circ g^{\vee} \circ \lambda_{A_2} \circ f \in \text{End}_{O_K}A_1 = O_K$$
By proposition 3.1.2 of \cite{Howard2012}, for $i= 1, 2$, every $(A,\kappa,\lambda) \in \cm_{\Phi}(\ktildebarp)$ has a unique deformation to $\wtildep$, we call this deformation the canonical lift of $(A,\kappa,\lambda)$. Assume the ramification conditions below:

1) $K/F$ is not unramified at all finite primes.

2) For every rational prime $l \leq \frac{[\tilde{K}:\mathbb{Q}]}{[K:\mathbb{Q}]}+1$, ramification index is less than $l$

\begin{proposition}
If $K/F$ is not unramified at all finite primes, then for any CM-type $\Phi$ of $K$, $\cm_{\Phi}(\mathbb{C})$ is nonempty.
\end{proposition}
\begin{proof}
Consider $\mathbb{C}^n$ constructed from the $\phi \in \Phi$. Let $\zeta \in K^{\times}$ satisfy $\bar{\zeta} = -\zeta$, also (using weak approximation theorem) assume that $\phi(\zeta)i > 0$ for all $\phi \in \Phi$. Now $\lambda(x,y) = \text{tr}_{K/\mathbb{Q}}(\zeta x \bar{y})$ is an $\mathbb{R}$-symplectic form on $K_{\mathbb{R}} \cong \mathbb{C}^n$ and $\lambda(ix,x)$ is positive definite. By class field theory and the ramification condition in the hypothesis, the norm map from ideal class group of $K$ to narrow ideal class group of $F$ is surjective, so there is a fractional $O_K$-ideal $I$ and $u \in F^{\gg 0}$ and that $uI\bar{I} = \zeta^{-1} \mathfrak{D}^{-1}$ where $\mathfrak{D}$ is the different of $K/\mathbb{Q}$. So by replacing $\zeta$ with $\zeta u^{-1}$ we have $\zeta I \bar{I} = \mathfrak{D}^{-1}$, so that 
$$I=\lbrace x \in K_{\mathbb{R}} | \lambda(x,I) \subseteq \mathbb{Z} \rbrace $$
Now $\lambda$ is a Riemann form on $K_{\mathbb{R}}/I$, so it is an abelian variety with the given CM-type.
\end{proof}

\begin{proposition}
Suppose that $k$ is an algebraically closed field and $(A_1,A_2) \in (\cm_{\Phi_1} \times_{\tilde{O}} \cm_{\Phi_2})(k)$, if $f \in V(A_1,A_2) := L(A_1,A_2) \otimes \mathbb{Q}$ is such that $\langle f,f \rangle \neq 0$, then $k$ has nonzero characteristic and $A_1, A_2$ are $O_K$-isogenous.
\end{proposition}
\label{prop3.2}
\begin{proof}
Suppose that $f \in L(A_1,A_2)$ and let $T_l(A_i)$ be the Tate module at a prime $l \neq \text{char} k$ and $T_l^{\circ}(A_i) := T_l(A_i) \otimes \mathbb{Q}$, consider the Weil pairing induced by polarizations:
$$i=1, 2; \;\;\;\Lambda_i: T_{l}^{\circ}(A_i) \times T_{l}^{\circ}(A_i) \rightarrow \mathbb{Q}_l(1)$$
By existence of $f$, we get a map $$f_l: T_{l}^{\circ}(A_1) \rightarrow T_{l}^{\circ}(A_2)$$
Consider the adjoint of $f_l$:
$$\Lambda_1(x,f_l^{\dag}y) = \Lambda_2(f_lx,y)$$
and we get that $f_l^{\dag} \circ f_l = \langle f,f \rangle$ as elements in $F_l \subseteq \text{End}(T_l^{\circ} A_1) \otimes \mathbb{Q}$, so that $f_l$
 is injective and an isogeny, so we have $O_K$-linear isogeny $A_1 \sim A_2$. Now, because CM-types of $A_1, A_2$ are different, $\text{char}\; k > 0$.
 \end{proof}
 
 From now on, we fix some $(A_1,A_2) \in (\cm_{\Phi_1} \times_{\tilde{O}} \cm_{\Phi_2})(\ktildebarp)$.
\begin{remark}
If $(\cm_{\Phi_1} \times_{\tilde{O}} \cm_{\Phi_2})(\ktildebarp)$ is nonempty with an $f$ as above, then $\mathfrak{p}_F$ is nonsplit in $K$. Indeed, On the level of Lie algebras, using the isogeny $f$, we must have $\phi^{sp} = \overline{\phi^{sp}}$, so that $x = \bar{x} \mod \mathfrak{p}$ for all $x \in K$, so $\mathfrak{p} = \bar{\mathfrak{p}}$, so that $\mathfrak{p}_K = \overline{\mathfrak{p}_K} $ and $\mathfrak{p}_F$ is nonsplit in $K$.
\end{remark}
\begin{proposition}
Let $k$ be an algebraically closed field and $\text{char } k > 0$ with $f \in L(A_1,A_2)$ such that $\langle f,f\rangle \neq 0$, then $L(A_1,A_2)$ is a projective $O_K$-module of rank 1. Let $q$ be a rational prime and $\mathfrak{q}$ above it in $F$, then the map
 $$L(A_1,A_2) \otimes_{O_{F}} O_{F,\mathfrak{q}} \rightarrow \text{Hom}_{O_{K}}(A_1[q^{\infty}],A_2[\mathfrak{q}^{\infty}])$$
is an isomorphism.
\end{proposition}
\begin{proof}
By the last proposition, we get an $O_K$-linear isogeny $A_2 \rightarrow A_1$, so we get
 $$\text{Hom}_{O_K}(A_1,A_2) \rightarrow \text{Hom}_{O_K}(A_1,A_1) \cong O_K$$
 which is injective and has finite cokernel,
 so $\text{Hom}_{O_K}(A_1,A_2)$ is a (nonzero) fractional ideal of $O_K$ (as $0 \neq f \in \Hom_{O_K}(A_1,A_2)$). Also, the map $A_2[q^{\infty}] \rightarrow A_1[q^{\infty}]$ will give us the injective map
 $$\text{Hom}_{O_K}(A_1[q^{\infty}],A_2[q^{\infty}]) \rightarrow \text{Hom}_{O_K}(A_1[q^\infty],A_1[q^\infty]) = O_K \otimes \mathbb{Z}_q$$
 with finite cokernel, so that $\text{Hom}_{O_K}(A_1[q^{\infty}],A_2[q^{\infty}]) \subseteq O_K \otimes \mathbb{Z}_q$ is a projective $O_K \otimes \mathbb{Z}_q$-module of rank 1.
 
 The map
 $$\text{Hom}_{O_K}(A_1,A_2) \otimes_{\mathbb{Z}} \mathbb{Z}_{q} \rightarrow \text{Hom}_{O_K}(A_1[q^{\infty}],A_2[q^{\infty}])$$
 is injective with $\mathbb{Z}_q$-torsion free cokernel and as the $\mathbb{Z}_q$-ranks of domain and codomain are equal (by the fact just mentioned), this is an isomorphism.
 \\
 So after taking $\mathfrak{q}$-parts, $L(A_1,A_2) \otimes_{O_F} O_{F,\mathfrak{q}} \rightarrow \text{Hom}_{O_K}(A_1[q^{\infty}],A_2[\mathfrak{q}^{\infty}])$
 is an isomorphism.
 \end{proof}

 Now consider the group $C_K = \mathcal{I}_K/\mathcal{P}_K$ where $\mathcal{I}_K$ consists of pairs  $(I,\zeta)$ such that $I$ is a fractional $O_K$-ideal and $\zeta \in F^{\gg 0}$ that $\zeta I \bar{I} = O_K$, multiplication is given by $(I_1,\zeta_1)(I_2,\zeta_2) = (I_1I_2, \zeta_1\zeta_2)$ and $\mathcal{P}_K = \lbrace (z^{-1}O_K, z\bar{z})| z \in K^{\times} \rbrace$ is a subgroup of it. Now $C_K$ acts on the set of all $(L,H)$ (where $L$ is an $O_K$-fractional ideal and $H$ is an $O_K$-valued $O_K$-Hermitian form on $L$) as follows: If $(I,\zeta) \in \mathcal{I}_K$, change this element by an element of $\mathcal{P}_K$ so that $\zeta \in O_K$ , and define the action by 
 $$(I,\zeta).(L,H) = (IL, \zeta H)$$
 Now for an abelian scheme $A \in \cm_{\Phi_i}(S)$ and for $I$ a fractional ideal of $O_K$, write $A^I$ for the Serre construction $I \otimes_{O_K} A \in \cm_{\Phi_i}(S)$.\\
 For a ring $R$, denote $R \otimes \hat{\mathbb{Z}}$ by $\hat{R}$ and for an $O_K$-module $M$, we denote $M \otimes_{O_K} \widehat{O}_K$ by $\hat{M}$ (so that $\widehat{L}(A_1,A_2) = L(A_1,A_2) \otimes_{O_K} \widehat{O}_K$
 Now we have
 \begin{proposition}
 Let $S$ be a connected $\tilde{O}$-scheme and $(A_1,A_2) \in (\cm_{\Phi_1} \times_{\tilde{O}} \cm_{\Phi_2})(\mathbb{C})$, then for any $(I,\zeta) \in \mathcal{I}_K$, we have an isomorphism of $\widehat{O}_K$-modules $\widehat{L}(A_1,A_2) \cong \widehat{L}(A_1,A_2^I)$ such that if $\langle\; ,\; \rangle^I$ is the Hermitian form of $\widehat{L}(A_1,A_2^{I})$, we have $$ \langle\; ,\; \rangle^{I} = \zeta z \bar{z}\langle\; ,\; \rangle$$ where $z$ is a finite idele of $K$ such that $z O_K = I$.
 \end{proposition}
 
 Now for $(A_1,A_2) \in (\cm_{\Phi_1} \times_{\tilde{O}} \cm_{\Phi_2})(\mathbb{C})$, define $L_{Betti}(A_1,A_2) = \Hom_{O_K}(H_1(A_1,\mathbb{C}),H_1(A_2,\mathbb{C}))$.
 
 \begin{proposition}
 Let $(A_1,A_2) \in (\cm_{\Phi_1} \times_{\tilde{O}} \cm_{\Phi_2})(\mathbb{C})$, then there exists $\beta \in \widehat{O}_F^{\times}$ and an isomorphism 
 $$(\widehat{L}_{Betti}(A_1,A_2),\langle\; , \;\rangle) \cong (\widehat{O}_K, \beta x\bar{y})$$
  Also, $\langle\; ,\; \rangle$ is negative definite at $\infty^{sp}$ and positive definite at others.
 \end{proposition}
 \begin{proof} We give a sketch below:
 
 We get the polarizations $\lambda_1(x,y) = \text{tr}_{K/\mathbb{Q}}(\xi x \bar{y})$, $\lambda_2(x,y) = \text{tr}_{K/\mathbb{Q}}(\zeta x \bar{y})$, then $\zeta \xi^{-1} O_K = O_K$. $L_{B}(A_1,A_2)$ has Hermitian form $\zeta \xi^{-1} x\bar{y}$, So that $\phi(\zeta \xi^{-1}) > 0$ for $\phi \neq \phi^{sp}, \phispbar$ and $\phi^{sp}(\zeta \xi^{-1})<0, \overline{\phi^{sp}}(\zeta \xi^{-1})<0$. \end{proof}
 Now consider all pairs $(L,H)$ such that\\
 $\bullet$ $L$ is a fractional ideal of $O_K$.\\
 $\bullet$ $H$ is an $O_K$-valued $O_K$-Hermitian form on $L$.\\
 $\bullet$ $(\widehat{L},H) \cong (O_K,\beta x \bar{y})$ with $\beta \in \widehat{O}_F^{\times}$.\\
 $\bullet$ $(L,H)$ has signature -1 at $\infty^{sp}$ and +1 at other embeddings.\\
 The set of these Hermitian spaces has a $C_K$-action and we have
 $$(I,\zeta).(L_{Betti}(A_1,A_2),\langle\; , \;\rangle) \cong (L_{Betti}(A_1,A_2^I),\langle \;, \;\rangle^I) $$\\
 From now on, we fix an isomorphism $\cp = \mathbb{C}$.
 \begin{theorem}
 Let $\mathfrak{p}$ be a prime of $\tilde{K}$ such that $\mathfrak{p}_F$ is nonsplit and $(A_1,A_2) \in (\cm_{\Phi_1} \times \cm_{\Phi_2})(\ktildebarp)$ and there exists $f \in L(A_1,A_2)$ such that $\langle f,f\rangle \in F^{\times}$, then 
 $$(\widehat{L}(A_1,A_2),\langle\; ,\; \rangle) \cong (\widehat{O}_K,\beta x\bar{y})$$ where $\beta \in \widehat{F}^{\times}$ and $\beta O_F = \mathfrak{p}_F^{\epf}$ where
  \[ \epf  =  \begin{cases} 
      1 & \text{if}\; K/F \; is \; \text{unramified} \\
      0 & \text{if}\; K/F\; is \; \text{ramified} \\ 
   \end{cases}
\]
 \end{theorem}
 \begin{proof}
 There exists a unique lift of $(A_1,A_2)$ to $\cp$ and then the isomorphism of $\tilde{K}$-algebras $\cp \cong \mathbb{C}$, we may view the unique lift as a pair $(A_1^{\prime},A_2^{\prime})$ in $(\cm_{\Phi_1} \times \cm_{\Phi_2})(\mathbb{C})$. Let $\mathfrak{q} \subseteq O_F$ be a prime with the rational prime $q$ below it, then there are isomorphisms of Hermitian $O_{K,\mathfrak{q}}$-modules :
 $$L_{Betti}(A_1^{\prime},A_2^{\prime}) \otimes_{O_K} O_{K,\mathfrak{q}} \cong \text{Hom}_{O_K}(A_1^{\prime}[q^\infty], A_2^{\prime}[\mathfrak{q^{\infty}}])$$
 $$L(A_1,A_2) \otimes_{O_K} O_{K,\mathfrak{q}} \cong \text{Hom}_{O_K}(A_1[q^\infty],A_2[\mathfrak{q^{\infty}}])$$
 Now we have the
 \begin{lemma}
 With the notation as above, let $\mathfrak{q} \subseteq O_F$ be a prime $\neq \mathfrak{p}_F$, we have $O_K$-linear isomorphisms
 $$\text{Hom}_{O_K}(A_1^{\prime}[q^\infty],A_2^{\prime}[\mathfrak{q}^\infty]) \cong \text{Hom}_{O_K}(A_1[q^\infty], A_2[\mathfrak{q}^{\infty}])$$
 respecting Hermitian forms.
 \end{lemma}
 \begin{proof}
 Let $q$ be the rational prime below $\mathfrak{q}$, If $q \neq p$, then the $q$-adic Tate modules of $A_i^{\prime}, A_i$ are isomorphic, so that it is true.\\
For $q=p$ and $i= 1, 2$, consider $\Phi_i(\mathfrak{q})$ to be all of $\phi \in \Phi_i$ such that the inverse image of maximal ideal of $O_{\cp}$ in
$K \hookrightarrow \cp$ is $\mathfrak{q}$. Now we have that $\Phi_1(\mathfrak{q}) = \Phi_2(\mathfrak{q})$ because $\phi^{sp}$ cannot be in $\Phi_1(\mathfrak{q})$ and similarly $\overline{\phi^{sp}}$ cannot be in $\Phi_2(\mathfrak{q})$ (as $\mathfrak{q}$ is not $\p_F$ by hypothesis). So by passing to p-divisible groups $A_i[\mathfrak{q}^{\infty}]$ with the action of $O_{K,\mathfrak{q}}$, we are in the situation of proposition 2.8 and letting $A_i^{can}$ be the canonical lifting of $A_i$ to $\tilde{\mathbb{W}}_{\p}$, we have
$$\Hom_{O_K}(A_1^{can}[p^{\infty}], A_2^{can}[\mathfrak{q}^{\infty}]) \rightarrow \Hom_{O_K}(A_1[p^{\infty}], A_2[\mathfrak{q}^{\infty}])$$
is an isomorphism, now by the base change $\wtildep \rightarrow \cp$ we get a map
\begin{equation}
\label{eq3.1}
\Hom_{O_K}(A_1^{can}[p^{\infty}], A_2^{can}[\mathfrak{q}^{\infty}]) \rightarrow \Hom_{O_K}(A_1^{\prime}[p^{\infty}], A_2^{\prime}[\mathfrak{q}^{\infty}])
\end{equation}
Now we have Tate's theorem that says for two $p$-divisible groups $G, H$ with Tate modules $TG, TH$ (over specific types of rings $R$, which includes $\wtildep$ and $\cp$ with $E = Frac(R)$) the map
$$\Hom(G,H) \rightarrow \Hom_{\text{Gal}(\bar{E}/E)}(TG, TH)$$
is an isomorphism. So that the map in equation \ref{eq3.1} is injective with image the submodule of $\text{Aut}(\cp/Frac(\wtildep))$-invariants of $\Hom_{O_K}(A_1^{\prime}[p^{\infty}], A_2^{\prime}[\mathfrak{q}^{\infty}])$, so that its cokernel is $\mathbb{Z}_p$-torsion free. Now both sides of the composite map $\Hom_{O_K}(A_1[p^{\infty}],A_2[\mathfrak{q}^{\infty}]) \rightarrow \Hom_{O_K}(A_1^{\prime}[p^{\infty}],A_2^{\prime}[\mathfrak{q}^{\infty}])$ are free of rank 1 over $O_{K,\mathfrak{q}}$, so that it is an isomorphism. It is obvious that the isomorphism respects Hermitian forms. 
 \end{proof}
 Now to prove the theorem, we just collect the propositions we proved to get:\
 If $\mathfrak{q} \neq \mathfrak{p}_F$, then 
 $$L(A_1,A_2) \otimes_{O_F} O_{F,\mathfrak{q}} \cong O_{K,\mathfrak{q}}$$ with Hermitian form $\beta_\mathfrak{q} x \bar{y} $ with $\beta_{\mathfrak{q}} \in O_{F,\mathfrak{q}}^{\times}$.\\
If $\mathfrak{q} = \mathfrak{p}_F$, then we use prop 2.3 to get 
 $\beta_\mathfrak{q} O_{F,\mathfrak{q}} = \mathfrak{p}_F^{\epf}$.\\
 Now set $\beta = \prod_{\mathfrak{q}} \beta_{\mathfrak{q}}$, so that the ideal of $L(A_1,A_2)$ is $\beta O_F = \mathfrak{p}_F^{\epf}$. So we have that $\widehat{L}(A_1,A_2)$ isomorphic to $\widehat{O}_K$ with Hermitian form given by $\beta x \bar{y}$. Now to prove $\chi_{K/F}(\beta)=1$, we have that $L(A_1,A_2) \otimes_{O_K} K$ is $K$ with the Hermitian form given by $c x \bar{y}$ with $c \in F^{\gg 0}$ and $\beta$ differs by norm with $c$ at each place of $K$, so that
 $$\chi_{K/F}(\beta) = \chi_{K/F}(c) = 1$$
 \end{proof}
Let $S$ be an $\tilde{O}$-scheme, then to each $(A_1,A_2) \in (\cm_{\Phi_1} \times_{\tilde{O}} \cm_{\Phi_2})(S)$, we have $(L(A_1,A_2),\langle\; ,\; \rangle)$.\\
Let $\mathcal{Z}(\alpha)$ (for $\alpha \in F^{\times}$ totally positive) be the algebraic stack that for an $\tilde{O}$-scheme $S$ assigns the groupoid of triples $(A_1,A_2,f)$, where $(A_1,A_2)\in (\cm_{\Phi_1} \times_{\tilde{O}} \cm_{\Phi_2})(S)$ and $f$ is an element of $L(A_1,A_2)$ with $\langle f,f \rangle = \alpha$.
\begin{proposition} Let $\alpha \in F^{\times}$,\\
1) Suppose that $\alpha \gg 0$, then the stack $\mathcal{Z}(\alpha)$ has dimension 0 and is supported in nonzero characteristics.\\
2) If $\mathfrak{p}$ is a prime in $\tilde{K}$ for which $\mathcal{Z}(\alpha)(\ktildebarp)$ is nonempty, then $\mathfrak{p}_F$ is nonsplit.
\end{proposition}
\begin{proof}
1) Second part of the statement comes from proposition 3.2. For the first part, we have that the map of stacks $\Z(\alpha) \rightarrow \cm_{\Phi_1} \times_{\tilde{O}} \cm_{\Phi_2}$ sending $(A_1,A_2,f)$ to $(A_1,A_2)$ is unramified, so that each of $\widehat{O}_{\Z(\alpha),z}$ is a quotient of $\wtildep$, but it can not be the whole $\wtildep$ because $\Z(\alpha)$ does not have points in characteristic 0, so it has to be of dimension 0, and so $\Z(\alpha)$ has dimension zero.\\
2) See Remark 3.3.
\end{proof}

Now we count the number of $\ktildebarp$ points of the stack $\Z(\alpha)$. For an fractional ideal $I$ of $O_F$ let $\rho(I)$ be the number of fractional ideals of $O_K$ with norm equal to $I$ in $F$, then
\begin{theorem}
Let $\alpha \in F^{\gg 0}$, assume that $K/F$ is not unramified at all finite primes, If $\mathfrak{p}$ is a prime such that $\mathfrak{p}_F$ is nonsplit, then
$$\# \Z(\alpha)(\ktildebarp) =\sum_{(A_1,A_2,f) \in Z(\alpha)(\ktildebarp)} \frac{1}{\# \text{Aut}(A_1,A_2,f)} = \frac{|C_K|}{w(K)} \rho(\alpha \mathfrak{p}_F^{-\epf} O_F)$$
\end{theorem}
\begin{proof}
First we define a subgroup of $C_K$ denoted $C_K^{\circ}$: Let $H$ be the algebraic group over $F$ defined for an $F$-algebra $R$ to be $H(R) = \ker({N}_{K/F} \otimes \text{id}: (K \otimes_F R)^{\times} \rightarrow (F \otimes_F R)^{\times})$ and a compact open subgroup of $\widehat{F}$-points of it, $U = \ker({N}_{K/F}: \widehat{O}_K^{\times} \rightarrow \widehat{O}_F^{\times}) \subseteq H(\widehat{F})$. Let $C_K^{\circ}$ be the double quotient (it is a finite group) $H(F) \backslash H(\widehat{F}) / U$. 
Also let $\eta: \widehat{O}_F^{\times}/{N}_{K/F}\widehat{O}_K^{\times} \rightarrow \lbrace \pm 1 \rbrace $ be the restriction of $\chi_{K/F}$ to $\widehat{O}_F^{\times}$.

Now we have the exact sequence
\begin{equation}
\label{eq3.2}
1 \rightarrow C_K^{\circ} \rightarrow C_K \rightarrow \widehat{O}_F^{\times}/\text{Norm}\widehat{O}_K^{\times} \xrightarrow{\eta} \lbrace \pm 1 \rbrace 
\end{equation}
For a set $S$, let $1_S$ be the characteristic function of $S$, we compute
\begin{align*}
&\sum_{I \in C_K^{\circ}} \# \lbrace f \in L(A_1,A_2^{I}) | \langle f,f \rangle^{I} = \alpha \rbrace \\
&= \sum_{I \in C_K^{\circ}} \sum_{\substack{x \in I.L(A_1,A_2) \otimes \mathbb{Q} \\ \langle x,x \rangle = \alpha}} 1_{I.L(A_1,A_2)}(x) \\
&= \sum_{h \in H(F) \backslash H(\widehat{F}) / U} \sum_{\substack{x \in L(A_1,A_2) \otimes \mathbb{Q} \\ \langle x,x \rangle = \alpha}} 1_{\widehat{L}(A_1,A_2)}(h^{-1}x) \\
&= \#(H(F) \cap U) \sum_{h \in H(\widehat{F})/U}
\sum_{\substack{x \in H(F)\backslash L(A_1,A_2) \otimes \mathbb{Q} \\ \langle x,x \rangle = \alpha}} 1_{\widehat{L}(A_1,A_2)}(h^{-1}x)
\end{align*}
Let $\mu(K)$ be the roots of unity in $K$, then for $z \in \Z(\alpha)(\ktildebarp)$, we have $Aut\; A \cong \mu(K)$, so that $Aut(A_1,A_2^{I}) \cong \mu(K)^2$, also we have $H(F) \cap U = \mu(K)$. Thus, we get
$$\sum_{I \in C_K^{\circ}} \sum_{\substack{f \in L(A_1,A_2^I) \\ \langle f,f \rangle^{I} = \alpha}} \frac{w(K)}{\# Aut(A_1,A_2)} = \sum_{h \in H(\widehat{F})/U} \sum_{\substack{x \in H(F) \backslash L(A_1,A_2) \otimes \mathbb{Q} \\ \langle x,x \rangle = \alpha}} 1_{\widehat{L}(A_1,A_2)}(h^{-1}x)$$
If there exists an $x \in H(F)\backslash L(A_1,A_2) \otimes \mathbb{Q}$ satisfying $\langle x,x \rangle = \alpha$, then we get a simply transitive action of $H(\widehat{F})$ on all such $x$ and so the last sum can be removed with fixing an $x \in L(A_1,A_2) \otimes \mathbb{Q}$ with the property $\langle x,x \rangle = \alpha$:
$$\sum_{I \in C_K^{\circ}} \sum_{\substack{f \in L(A_1,A_2^I) \\ \langle f,f \rangle^{I} = \alpha}} \frac{w(K)}{\# Aut(A_1,A_2)} = \frac{1}{w(K)} \sum_{h \in H(\widehat{F})/U} 1_{\widehat{L}(A_1,A_2)}(h^{-1}x)$$

Summing above over $I \in C_K / C_K^{\circ}$ and using exactness of equation (\ref{eq3.2}), we have
\begin{equation}
\label{eq3.3}
\sum_{I \in C_K} \sum_{\substack{f \in L(A_1,A_2^{I})\\ \langle f,f \rangle^I = \alpha}} \frac{1}{\# Aut(A_1,A_2^I)} = \frac{1}{w(K)} \sum_{\xi \in \ker \eta} \sum_{h \in H(\widehat{F})/U} 1_{\widehat{L}(A_1,A_2)}(h^{-1}x_{\xi})
\end{equation}
where $x_{\xi}$ is an element of $L(A_1,A_2) \otimes \widehat{\mathbb{Q}}$ such that $\langle x_{\xi}, x_{\xi} \rangle = \xi \alpha$. Now let's fix $x \in \widehat{K}$ such that $\alpha = \beta x \bar{x}$ (recall that $\widehat{L}(A_1,A_2)$ has Hermitian form $\beta x \bar{y}$) and compute the orbital integral
$$O_{\alpha}(A_1,A_2) = \sum_{h \in  H(\widehat{F})/U} 1_{\widehat{L}(A_1,A_2)}(h^{-1}x)$$
Also define the local orbital integrals 
$$O_{\alpha,v}(A_1,A_2) = \sum_{h \in H({F_v})/U_v} 1_{O_{K,v}}(h^{-1}x_v)$$
Thus we can factor $O_{\alpha}(A_1,A_2)$ as 
$$O_{\alpha}(A_1,A_2) = \prod_{v} O_{\alpha,v}(A_1,A_2)$$
where the product is over the finite places of $F$. If $v$ is nonsplit in $K$, then
 \[ O_{\alpha,v}(A_1,A_2)  =  \begin{cases} 
      1 & \text{if}\; \alpha\beta^{-1} \in O_{F,v} \\
      0 & \text{}\; \text{otherwise} \\ 
   \end{cases}
\]
and if $v$ is split in $K$, then $K_v \cong F_v \otimes F_v$ and by fixing a uniformizer $\pi \in F_v$, we find that $H(F_v)/U_v$ is the cyclic group generated by $(\pi,\pi^{-1}) \in F_v^{\times} \times F_v^{\times}$ and
 \[ O_{\alpha,v}(A_1,A_2)  =  \begin{cases} 
      1+\text{ord}_{v}(\alpha_v \beta_v^{-1}) & \text{if}\; \alpha\beta^{-1} \in O_{F,v} \\
      0 & \text{if}\; \text{otherwise} \\ 
   \end{cases}
\]
So that $O_{\alpha,v}(A_1,A_2)$ is the number of ideals $\mathfrak{a} \subseteq O_{K,v}$ such that
$$\beta_v \mathfrak{a}_v \bar{\mathfrak{a}}_v = \alpha O_{F,v}$$
and therefore $O_{\alpha}(A_1,A_2) = \rho(\alpha \beta^{-1} O_F)$.

Now in the RHS of equation (\ref{eq3.3}) we have that there is a unique $\xi \in \ker \eta$ such that the sum is not zero (because we must have that $\rho(\alpha \beta^{-1} O_F) \neq 0$ and this means that there exists an ideal $\mathfrak{a} \subseteq O_K$ such that $\alpha O_F = \beta \mathfrak{a} \bar{\mathfrak{a}}$ and so this means that there's a unique $\xi \in \widehat{O}_F^{\times}/\text{Norm}(\widehat{O}_K^{\times})$ that $\xi \alpha$ is represented by $\beta x \bar{x}$, also as $\chi_{K/F}(\beta)=1$, we get $\xi \in \ker \eta$). Thus, equation (\ref{eq3.3}) is
$$\sum_{\xi \in \ker \eta} O_{\xi \alpha}(A_1,A_2) = \rho(\alpha\beta^{-1}O_F) = \rho(\alpha \p_F^{-\epsilon_{\p_F}})$$
Therefore as $C_K$ acts simply transitively on $\cm_{\Phi}(\ktildebarp)$  for a CM-type $\Phi$ of $K$, we get
$$\sum_{\substack{A_1 \in \cm_{\Phi_1}(\ktildebarp) \\ A_2 \in \cm_{\Phi_2}(\ktildebarp)}} \sum_{\substack{f \in L(A_1,A_2) \\ \langle f,f \rangle = \alpha}} \frac{1}{\# Aut(A_1,A_2)} = \frac{|C_K|}{w(K)}\rho(\alpha \p^{-\epsilon_{\p_F}})$$
which completes the proof.
\end{proof}

Now we have
\begin{theorem}
\label{thm3.11}
Let $\alpha \in F^{\gg 0}$, $\mathfrak{p}$ be a prime of $\tilde{K}$ such that $\mathfrak{p}_F$ is nonsplit in $K$, then at $z \in \mathcal{Z}(\alpha)(\ktildebarp)$:
$$length(O^{\text{\'{e}}t}_{\mathcal{Z}(\alpha),z}) = \frac{1}{2} e_{\mathfrak{p}_F}(\text{ord}_{\mathfrak{p}_F}(\alpha) + 1)$$
where $e_{\mathfrak{p}_F}$ is the ramification index of $\phispec(\mathfrak{p}_F)$ in $\tilde{K}/F$. 
\end{theorem}
\begin{proof}
We have the decomposition \[ A[p^{\infty}] = \prod_{\substack{\mathfrak{q} \subseteq O_F \\ \mathfrak{q}|p}} A[\mathfrak{q}^{\infty}]\] Now we have a map $f:A_1 \rightarrow A_2$ and by Serre-Tate, the deformation functor of $(A_1, A_2, f)$ is the same as the deformation functor of 
$(A_1[p^{\infty}], A_2[p^{\infty}], f[p^\infty]: A_1[p^\infty] \rightarrow A_2[p^\infty]$), now we have the decompositions
$$f[\mathfrak{q}^{\infty}]: A_1[\mathfrak{q}^\infty] \rightarrow A_2[\mathfrak{q}^\infty]$$
Case 1) $\mathfrak{q} \neq \mathfrak{p}_F$, then we can use proposition \ref{prop2.8} to prove that the points $(A_1[\mathfrak{q}^{\infty}],A_2[\mathfrak{q}^{\infty}],f[\mathfrak{q}^\infty])$ can always be deformed to objects of $\art$. \\
Case 2) $\mathfrak{q} = \mathfrak{p}_F$, we get that the deformation functor of $(A_1[\mathfrak{p}_F^{\infty}],A_2[\mathfrak{p}_F^{\infty}],f[\mathfrak{p}_F^{\infty}])$ is pro-represented by $\tilde{\W}_{\mathfrak{p}}/\m^k$ ($m$ is the maximal of $\tilde{\W}_{\mathfrak{p}}$) where $k = \frac{1}{2}\text{ord}_{\mathfrak{p}}(\alpha \mathfrak{p}_F)= \frac{1}{2} e_{\mathfrak{p}_F} (\text{ord}_{\mathfrak{p}_F}(\alpha) + 1)$ by Theorem \ref{Thm2.4}.
\end{proof}
Collecting everything together we get the concluding theorem:
\begin{theorem}
\label{thm3.12}
 We have
$$\widehat{\deg}\mathcal{Z}(\alpha) = \frac{|C_K|}{w(K)} \sum_{\mathfrak{p} \subset O_F} \frac{\log N(\mathfrak{p})}{[K:\mathbb{Q}]}(\ord_{\mathfrak{p}}(\alpha)+1)\rho(\alpha \mathfrak{p}^{-e_{\mathfrak{p}}})$$ where $\mathfrak{p}$ goes over the nonsplit primes of $F$.
\end{theorem}
\begin{proof}
Let $f_{\p/\p_F}$ be the residue degree of $\p$ over $\p_F$. Now by theorem 3.10 and 3.11, we have
\begin{align*}
\widehat{\deg} \Z(\alpha) &= \frac{1}{[\tilde{K}:\mathbb{Q}]} \sum_{\p \subset \tilde{O}} \log\; N(\p) \sum_{z \in \Z(\alpha)(\ktildebarp)} 
\frac{\text{length}(O^{\text{\'et}}_{\Z(\alpha),z})}{ \# \text{Aut}\; z} \\
& = \frac{1}{[\tilde{K}:\mathbb{Q}]}\sum_{\p \subset \tilde{O}}{}^{\prime} \log\; N(\p) \frac{e_{\p_F}(\text{ord}_{\p_F}(\alpha) + 1)|C_K|\rho(\alpha \p_F^{-\epsilon_{\p_F}}O_F)}{2w(K)} \\
& = \frac{|C_K|}{2w(K)[\tilde{K}:\mathbb{Q}]}\sum_{\p \subset \tilde{O}}{}^{\prime} \log\; N(\p_F) e_{\p_F}f_{\p/\p_F}(\text{ord}_{\p_F}(\alpha) + 1)\rho(\alpha \p_F^{-\epsilon_{\p_F}}O_F) \\
& = \frac{|C_K|}{2w(K)[\tilde{K}:\mathbb{Q}]}\sum_{\p \subset O_F}{}^{\prime} \log\; N(\p) [\tilde{K}:F](\text{ord}_{\p}(\alpha) + 1)\rho(\alpha \p^{-\epsilon_{\p}}O_F) \\
& = \frac{|C_K|}{w(K)} \sum_{\mathfrak{p} \subset O_F}{}^{\prime} \frac{\log N(\mathfrak{p})}{[K:\mathbb{Q}]}(\ord_{\mathfrak{p}}(\alpha)+1)\rho(\alpha \mathfrak{p}^{-e_{\mathfrak{p}}})
\end{align*}
where prime over sigma means that we are considering primes $\p$ in $\tilde{K}$ such that $\p_F$ is nonsplit in $K$ in the first two sums, and primes in $F$ that are nonsplit in $K$ in the last two sums. This finishes the proof of the theorem.
\end{proof}

\section{Arithmetic Chow group}
An arithmetic divisor on $\cm_{\Phi_1} \times_{\tilde{O}} \cm_{\Phi_2}$ is a pair 
$(\text{Z},\text{Gr})$ where Z is a divisor and Gr a Green function for Z. As $\mathcal{Z}(\alpha)$  does not have characteristic zero points, the Green function for $\mathcal{Z}(\alpha)$ can be any complex-valued function on the finite set 
$$\coprod_{\sigma: \tilde{K} \rightarrow \mathbb{C}}  (\cm_{\Phi_1} \times_{\tilde{O}} \cm_{\Phi_2})^{\sigma}(\mathbb{C})$$
Recall that for a scheme $S \rightarrow \text{Spec}\; \tilde{K}$ and a map $\sigma: \tilde{K} \rightarrow \mathbb{C}$, $S^{\sigma}$ is obtained by the base change along the map
$\text{Spec}\;\mathbb{C} \rightarrow \text{Spec}\; \tilde{K}$.  We define the Green's function for $\mathcal{Z}(\alpha)$ first for $(\cm_{\Phi_1} \times_{\tilde{O}} \cm_{\Phi_2})(\mathbb{C})$ at
$(A_1,A_2) = z \in (\cm_{\Phi_1} \times_{\tilde{O}} \cm_{\Phi_2})(\mathbb{C})$ to be
$$\text{Gr}_{\alpha}(y,z) = \sum_{\substack{f \in L_{Betti}(A_1,A_2)\\\langle f,f \rangle = \alpha}} \beta_1(4\pi |y\alpha|_{\infty^{sp}})$$
where $\beta_1: \mathbb{R}^{>0} \rightarrow \mathbb{R}$ is $\beta_{1}(t) = \int_{1}^{\infty} e^{-tu} \frac{du}{u}$. Also, to define $\text{Gr}_{\alpha}$ on $z \in (\cm_{\Phi_1} \times_{\tilde{O}} \cm_{\Phi_2})^{\sigma}(\mathbb{C})$, first we extend $\sigma: \tilde{K} \rightarrow \mathbb{C}$ to $\tilde{\sigma} \in \text{Aut}(\mathbb{C})$ and we obtain two new CM-types $\tilde{\sigma} \circ \Phi_1$, $\tilde{\sigma} \circ \Phi_2$ and then 
$$(\cm_{\Phi_1} \times_{\tilde{O}} \cm_{\Phi_2})^{\sigma}(\mathbb{C}) = (\cm_{\tilde{\sigma} \circ \Phi_1} \times_{\tilde{O}} \cm_{\tilde{\sigma} \circ\Phi_2})$$ 
and then we take the $\text{Gr}_{\alpha}(y,z)$ for $(A_1,A_2) = z \in (\cm_{\Phi_1} \times_{O_{\Phi}} \cm_{\Phi_2})^{\sigma}(\mathbb{C})$ to be 
$$\sum_{\substack{f \in L_{Betti}(A_1,A_2)\\\langle f,f \rangle = \alpha}} \beta_1(4\pi |y\alpha|_{\tilde{\sigma} \circ \infty^{sp}})$$
Now as $\langle\; ,\; \rangle$ is negative definite  at $\infty^{sp}$ and positive definite at other places, we have $\text(Gr)_{\alpha} = 0$ unless $\alpha$ is negative definite at exactly one place. So, we define 
$$\widehat{\mathcal{Z}}(\alpha) = (\mathcal{Z}(\alpha),\text{Gr}_\alpha) \in \widehat{\text{CH}}^1(\cm_{\Phi_1} \times_{\tilde{O}} \cm_{\Phi_2})$$
where for $\alpha$ not totally positive, $\mathcal{Z}(\alpha) = 0$ in $\text{CH}^1(\cm_{\Phi_1} \times_{\tilde{O}} \cm_{\Phi_2})$.

Now the degree of an arithmetic divisor is defined by the composition of the following maps:
$$\widehat{\deg}: \widehat{\text{CH}}^1(\cm_{\Phi_1} \times_{\tilde{O}} \cm_{\Phi_2}) \rightarrow \widehat{\text{CH}}^1(\text{Spec}\; \tilde{O}) \rightarrow \mathbb{R}$$
where left hand side map is the push forward of $(\cm_{\Phi_1} \times_{\tilde{O}} \cm_{\Phi_2}) \rightarrow \text{Spec}\; \tilde{O}$. Then, we would have
$$\widehat{\deg}(Z,\text{Gr}) = \frac{1}{[\tilde{K}:\mathbb{Q}]}(\sum_{\mathfrak{p} \subset \tilde{O}} \sum_{z \in \Z(\ktildebarp)} \frac{\log N(\mathfrak{p})}{\# \text{Aut}\; z} + \sum_{\sigma: \tilde{K} \rightarrow \mathbb{C}}\sum_{z \in (\cm_{\Phi_1} \times_{\tilde{O}} \cm_{\Phi_2})^{\sigma}(\mathbb{C})}\frac{\text{Gr(z)}}{\# \text{Aut}\; z})$$
Then by Theorem \ref{thm3.12} and the above discussion, we have 
$$\widehat{\text{deg}}\widehat{Z}(\alpha) = 
\begin{cases} 
\frac{|C_K|}{w(K)} \sum_{\mathfrak{p}} \frac{\log N(\mathfrak{p})}{[K:\mathbb{Q}]}(\text{ord}_{\mathfrak{p}}(\alpha)+1)\rho(\alpha \mathfrak{p}^{-\epsilon_\mathfrak{p}}) & \text{if}\; \alpha \gg 0 \\
\frac{|C_K|}{w(K)}\frac{\rho((\alpha))}{[K:\mathbb{Q}]}\beta_1(4\pi|y\alpha|_{v}) & \alpha\; \text{is negative definite at exactly one place v} \\
0 & \text{otherwise}
\end{cases}$$
\section{Eisenstein series}

In this section, we will construct an Eisenstein series for $\text{SL}_2(F)$ such that the degree of the divisors and arithmetic divisors, as computed in the preceding section, is up to an explicit multiplicative constant equal to the Fourier expansion of the central value of the derivative of this Eisenstein series.

Let $v$ be a place of $F$,
$\psi_F = \psi_\mathbb{Q} \circ \text{tr}_{F/\mathbb{Q}}$ an additive character of $\mathbb{A}_F/F$, where $\psi_{\mathbb{Q}}$ is the canonical additive character of $\mathbb{A}_{\mathbb{Q}}/\mathbb{Q}$,
$\psi_{F_v} = \psi_F|_{F_v}$,
$\mathbb{H}_F = \lbrace x+iy \in F_{\mathbb{C}} | y \gg 0 \rbrace$,
for $s \in \mathbb{C}$, $g \in \text{SL}_2(F_v)$, $x \in F_v^{\times}$, let $\Phi_{x,\psi_{F_v}}(g,s) \in I(\chi_v,s)$ be the same section defined in \cite{Howard2012}, and $\Phi_{c,\psi_F} = \otimes_v \Phi_{c,\psi_{F_v}}$,
$\psi$ a (general) additive character of $F_v$,
for simplicity, $\chi$ be the character of $\mathbb{A}_F^{\times}$ associated to the quadratic extension $K/F$ and $\chi_v$ the character of $F_v^{\times}$ associated to $K_v/F_v$,
$L(s,\chi) = \prod_{w} L(s,\chi_w)$ where the product is over all places (including archimedean ones) $w$ of $F$ (If $w$ is archimedean, let
$L(s,\chi_w) = \pi^{-\frac{s+1}{2}} \Gamma(\frac{s+1}{2})$),
$\pi_{F_v}$ be a uniformizer of $F_v$ and $\mathfrak{p}_v$ the prime ideal of $v$ in $F$,
for $I$ an ideal of $O_F$, $N(I) = \#(O_{F}/I)$ as usual.

For $c \in \mathbb{A}_F^{\times}$, $g \in \text{SL}_2(F_v)$, We construct an Eisenstein series $$E(g,s,c,\psi_F) := \sum_{\gamma \in \text{B}(F)\textbackslash\text{SL}_2(F)} \Phi_{c,\psi_F}(\gamma g,s)$$
where $B$ is the Borel subgroup of $\text{SL}_2$.
For $\tau = x+iy \in \mathbb{H}_F$, take 
$$g_{\tau} = 
\begin{bmatrix}
    1 & x \\
    0 & 1
\end{bmatrix}
\begin{bmatrix}
y^{\frac{1}{2}} & 0 \\
0 & y^{-\frac{1}{2}}
\end{bmatrix}
$$
at the archimedean places and identity at finite places.\\
Now define 
$$\mathcal{E}(\tau, s, c, \psi_F) = \frac{L(s+1,\chi)}{N_{F/\mathbb{Q}}(y)^{\frac{1}{2}}} E(g_{\tau},s,c,\psi_F)$$
Also, we define the local Whittaker function at $v$ to be ($\alpha \in F_v^{\times}, g \in \text{SL}_2(F_v)$):
$$W_{\alpha}(g,s,c,\psi) = \int_{F_v} \Phi_{c_v,\psi}(\wmatrix \nmatrix g,s)\psi(-\alpha x)dx$$

where $dx$ is the Haar measure on $F_v$ with respect to $\psi$.\\
If $\delta \in F_v^{\times}$, we set $(\delta \psi)(x) = \psi(\delta x)$, then 
$\Phi_{c,\delta \psi} = \Phi_{\delta c, \psi}$, then
\begin{equation}
W_{\alpha}(g,s,c,\delta \psi) = |\delta|_v^{\frac{1}{2}} W_{\delta \alpha}(g,s,\delta c, \psi)
\end{equation}
Now $\mathcal{E}(\tau,s,c,\psi_F)$ is a Hilbert modular form and has a Fourier expansion
$$\mathcal{E}(\tau, s, c, \psi_{F}) = \sum_{\alpha \in F} \mathcal{E}_{\alpha}(\tau,s,c, \psi_F)$$
Now suppose that $c$ is a unit at finite places and $c_v = 1$ for all archimedean places and $\chi(c) = -1$, then $\mathcal{E}$ has parallel weight 1 and is incoherent, so $\mathcal{E}(\tau, 0, c, \psi_F)= 0$.\\
Now define the final Eisenstein series to be the following sum 
$$\mathcal{E}_{\Phi}(\tau, s) = \sum_{c \in \Xi } \mathcal{E}(\tau,s,c,\psi_F)$$
Where $\Xi$ is the $N_{K/F}(\widehat{O}_K^{\times})$-orbits of $c$ satisfying:\\
1) $c$ is a unit at finite places.\\
2) $c_v = 1$ for all archimedean places.\\
3) $\chi(c) = -1$.\\
Define the difference set $\text{Diff}(\alpha,c) = \lbrace v | \chi_v(\alpha c) = -1 \rbrace$ (By the condition $\chi(c) = -1$ and the product formula, this set has odd cardinality, also every $v \in \text{Diff}(\alpha,c)$ is nonsplit in $K$), also as before set 
$$\rho_v(I) = \# \lbrace J \subseteq O_{K,v} | J\bar{J} = IO_{K,v} \rbrace $$
We have that $\rho$ is a multiplicative function.\\
Now we compute the Fourier coefficients at $s=0$ of the derivative of $\mathcal{E}_{\alpha}$.
\begin{proposition}
\label{prop5.1}
Suppose that the following ramification conditions are satisfied:

1) $K/F$ is not unramified at all finite primes.

2) For every rational prime $l \leq \frac{[\tilde{K}:\mathbb{Q}]}{[K:\mathbb{Q}]}+1$, ramification index is less than $l$.

$\alpha \in F^{\times}, d_{K/F}$ be the relative discriminant of $K/F$, $r$ the number of places of $F$ (including archimedean places) ramified in $K$ and $c \in \Xi$, then:\\
1) $\ord_{s=0} \mathcal{E}_{\alpha}(\tau,s,c,\psi_F) \geq \# \emph{Diff}(\alpha,c) $ so if $\# \emph{Diff}(\alpha,c) \geq 2$, then $ \frac{d}{ds} \mathcal{E}_{\alpha}(\tau,s,c,\psi_F)|_{s=0} = 0$.\\
2) If $\emph{Diff}(\alpha,c) = \lbrace \mathfrak{p} \rbrace$  with $\mathfrak{p}$ finite, then 
$$\frac{d}{ds}\mathcal{E}_{\alpha}(\tau,s,c,\psi_F)|_{s=0} = \frac{-2^{r-1}}{N(d_{K/F})^{\frac{1}{2}}} \rho(\alpha \mathfrak{p}^{-e_{\mathfrak{p}}}) (\ord _{\mathfrak{p}}\;\alpha + 1)\log N(\mathfrak{p}) q^{\alpha}$$
where $q^{\alpha}$ means $e^{2\pi i \text{tr}_{F/\mathbb{Q}}(\alpha \tau)}$, and $e_{\mathfrak{p}}$ is, as before, 0 if $\mathfrak{p}$ is unramified in $K$ and is 1 otherwise.\\
3) If $\emph{Diff}(\alpha,c) = \lbrace v \rbrace$ for $v$ archimedean place of $F$, then
$$\frac{d}{ds}\mathcal{E}_{\alpha}(\tau,s,c,\psi_F)|_{s=0} = \frac{-2^{r-1}}{N(d_{K/F})^{\frac{1}{2}}}\rho(({\alpha})) \beta_{1}(4\pi|y\alpha|_{v})q^{\alpha}$$ 
where $\beta_1$ is defined in section 4.
\end{proposition}
\begin{proof}
For $g \in \text{SL}_2(F_v)$, consider the normalized local Whittaker function
$$W_{\alpha_v}^{\ast}(g_v, s, c_v, \psi_{F_v}) = L(s+1,\chi_v) W_{\alpha_v}(g_v,s,c_v,\psi_{F_v})$$
Now we have the factorization of Fourier coefficients of $\mathcal{E}$:
$$\mathcal{E}_{\alpha}(\tau,s,c,\psi_F) = N_{F/\mathbb{Q}}(y)^{-\frac{1}{2}}\prod_{v} W_{\alpha_v}^{\ast}(g_{\tau,v},s,c_v,\psi_{F_v})$$
So by equation (3), we have
$$\mathcal{E}_{\alpha}(\tau,s,c,\psi_F) =N_{F/\mathbb{Q}}(y)^{-\frac{1}{2}} \prod_{v} W_{c_v^{-1}\alpha_v}^{\ast}(g_{\tau,v},s,1,c_v\psi_{F_v})$$
where $(c\psi_{F})(x) = \psi_{F}(c x)$ is an unramified character of $\mathbb{A}_F^{\times}$ and $(c_v\psi_{F_v})(x) = \psi_{F_v}(c_v x)$.\\
Now by the formulas of Yang, stated in prop. 2.1, 2.2, 2.3, 2.4 of \cite{Yang2013}, we have the following:\\
Suppose that $v$ is a finite prime of $F$, then 
$$\chi_v(\alpha c) = -1 \Leftrightarrow W_{c_v^{-1} \alpha_v}^{\ast}(g_{\tau, v},0,1,c_v\psi_{F_v}) = 0$$
(Recall that for a finite prime $v$, we have $g_{\tau,v} = I \in \text{SL}_2(F_v)$), so, we have the two cases:\\
1) $\chi_v(\alpha c)= 1$, then 
$$W_{c_v^{-1} \alpha_v}^{\ast}(g_{\tau, v},0,1,c_v\psi_{F_v}) =$$ $$= \chi_v(-1) \epsilon(\frac{1}{2},\chi_v,c_v\psi_{F_v})\rho_v(\alpha O_{F_v})
\begin{cases} 2N(\pi_{F_v})^{-\frac{\ord_v(d_{K/F})}{2}} & \text{if} \; v \; \text{ramified}\; \text{in} \; K/F \\
1 & \text{if} \; v \; \text{unramified}\; \text{in} \; K/F
\end{cases}$$
\\
2) $\chi_v(\alpha c)= -1$, then 
$$
\frac{d}{ds} W_{c_v^{-1}\alpha_v}^{\ast}(g_{\tau,v},s,1,c_v\psi_{F_v})|_{s=0} = $$
$$
= \chi_v(-1)\epsilon(\frac{1}{2},\chi_v,c_v\psi_{F_v})\log|\pi_{F_v}|^{-1} \frac{\ord_v(\alpha)+1}{2}
\begin{cases}
2N(\pi_{F_v})^{-\frac{\ord_v(d_{K/F})}{2}} \rho_v(\alpha O_{F_v}) & \text{if} \; v \; \text{ramified}\; \text{in} \; K/F \\
\rho_v(\alpha \mathfrak{p}_v^{-1}) & \text{if} \; v \; \text{unramified}\; \text{in} \; K/F
\end{cases}
$$
The point here is that :\\
1) If $v$ is ramified, $c_v^{-1} \alpha_v \in O_{F_v}$ iff $\rho_v(\alpha_v O_{F_v}) = 1$, so $\rho_v(\alpha_v O_{F_v}) = 1_{O_{F_v}}(c_v^{-1}\alpha_v)$.\\
2) $f = \ord_vd_{K/F} = 1$ (in Yang's formula prop. 2.3 (2), if $p \neq 2$ which is the case in here by our ramification condition) and also if $K_v/F_v$ is unramified, $\chi_v(\alpha_v c_v^{-1}) = -1 \Leftrightarrow K_v/F_v$ is inert and $\ord_{v}\alpha$ is odd (so we need $\mathfrak{p}_v^{-1}$ to make $\alpha \mathfrak{p}_v^{-1}$ have even order to get $\rho_v(\alpha \mathfrak{p}_v^{-1}) = 1_{O_{F_v}}(\alpha_v c_v^{-1})$).\\
3) $\chi_v(-1)\epsilon(\frac{1}{2},\chi_v,c_v\psi_{F_v}) = 1$ for $v$ unramified by prop 2.8 of \cite{kudla2003}.\\
4) If $v$ is split, then $\rho_v(\alpha_v O_{F_v}) = \ord_v\alpha + 1$ and if $v$ is inert, then $\rho_v(\alpha_v O_{F_v}) = \frac{1+(-1)^{\ord_v\alpha}}{2}$.\\
Also for an archimedean place $v$ of $F$, we have 
$$W_{c_v^{-1}\alpha_v}^{\ast}(g_{\tau,v},0,1,c\psi_{F_v}) = 2\chi_v(-1) \epsilon(\frac{1}{2},\chi_v,c\psi_{F_v})y_v^{\frac{1}{2}}e^{2\pi i\alpha_v \tau_v}$$
If $\chi_v(\alpha c) = -1$ then $W_{c_v^{-1}\alpha_v}^{\ast}(g_{\tau,v},0,1,c_v\psi_{F_v})=0$ and
$$\frac{d}{ds} W_{c_v^{-1}\alpha_v}^{\ast}(g_{\tau,v},s,1,c_v\psi_{F_v})|_{s=0} = 
\chi_v(-1)\epsilon(\frac{1}{2},\chi_v,c\psi_{F_v})y_v^{\frac{1}{2}}e^{2\pi i\alpha_v \tau_v}\beta_1(4\pi|y\alpha|_v)$$
Now if $\text{Diff}(\alpha,c) = \lbrace w \rbrace$, then
$$\frac{d}{ds} \mathcal{E}_{\alpha}(\tau,s,c,\psi_F)|_{s=0} = N_{F/\mathbb{Q}}(y)^{-\frac{1}{2}} \frac{d}{ds} W_{c_w^{-1}\alpha_w}^{\ast}(g_{\tau,w},s,1,c\psi_{F_w})|_{s=0} \prod_{v \neq w}W_{c_v^{-1}\alpha_v}^{\ast}(g_{\tau,v},0,1,c\psi_{F_v})$$
So by the aforementioned formulas, we obtain the formulas stated in the theorem by using the fact that
$$\prod_v\epsilon(\frac{1}{2},\chi_v,c\psi_{F_v}) = \chi(c) \prod_{v}\epsilon(\frac{1}{2},\chi_v,\psi_{F_v})= (-1)(1) = -1$$
\end{proof}
Now we have the Fourier expansion 
$$\mathcal{E}_{\Phi} = \sum_{\alpha \in F} \mathcal{E}_{\Phi,\alpha}$$
for $\mathcal{E}_{\Phi}(\tau,s) = \sum_{c \in \Xi} \mathcal{E}_{\alpha}(\tau,s,c,\psi_F)$, so that
$$\frac{d}{ds} \mathcal{E}_{\Phi,\alpha}|_{s=0} = \sum_{c \in \Xi} \mathcal{E}_{\alpha}^{\prime}(\tau,s,c,\psi_F)\frac{-2^{r-1}}{N(d_{K/F})^{\frac{1}{2}}}\sum_{\substack{\mathfrak{p}\subseteq O_F\; \\ \text{nonsplit}}} (\ord_{\mathfrak{p}}\alpha + 1)(\rho(\alpha \mathfrak{p}^{-e_{\mathfrak{p}}}))(\log N(\mathfrak{p}))q^{\alpha} =: b_{\Phi}(\alpha,y)$$
So we get to the main theorem:
\begin{theorem}
\label{Thm5.2}
Suppose that conditions of proposition \ref{prop5.1} are satisfied, then we have
$$\widehat{\deg}\widehat{\mathcal{Z}}(\alpha) = \frac{-|C_K|}{w(K)} \frac{\sqrt{N_{F/\mathbb{Q}}(d_{K/F})}}{2^{r-1} [K:\mathbb{Q}]} b_{\Phi}(\alpha,y)$$
\end{theorem}








\end{document}